\documentclass[11pt]{article}

\usepackage[T1]{fontenc}
\usepackage[utf8]{inputenc}
\usepackage{lmodern}
\usepackage{amsmath,amssymb,amsthm}
\usepackage{graphicx}
\usepackage{float}
\usepackage{hyperref}
\usepackage[margin=1in]{geometry}

\hypersetup{
  colorlinks=true,
  linkcolor=blue,
  citecolor=blue,
  urlcolor=blue
}

\newtheorem{theorem}{Theorem}[section]
\newtheorem{proposition}[theorem]{Proposition}
\newtheorem{lemma}[theorem]{Lemma}

\newtheorem{question}[theorem]{Question}
\theoremstyle{definition}
\newtheorem{definition}[theorem]{Definition}
\newtheorem{example}[theorem]{Example}
\theoremstyle{remark}
\newtheorem{remark}[theorem]{Remark}

\title{Quantitative Merging for Time Inhomogeneous Markov Chains in Non-Decreasing Environments via Functional Inequalities}
\author{Nordine Anis Moumeni\\
Aix-Marseille Universit\'e, CNRS, Centrale Marseille, I2M, UMR 7373, Marseille, France\\
\texttt{nordine.moumeni@univ-amu.fr}}
\date{}

\begin{document}

\maketitle

\begin{abstract}
We study time-inhomogeneous Markov chains to obtain quantitative results on their asymptotic behavior. We use Poincar\'e, Nash, and logarithmic-Sobolev inequalities. We assume that our Markov chain admits a finite invariant measure at each time and that the sequence of these invariant measures is non-decreasing. We deduce quantitative bounds on the merging time of the distributions for the chain started at two arbitrary points and we illustrate these new results with examples.
\end{abstract}

\noindent\textbf{Keywords.} Merging; Markov chains; Time-inhomogeneous; Functional inequalities; Non-decreasing environments.

\noindent\textbf{Mathematics Subject Classification.} NA.

\section{Introduction}\label{Introdu}
\subsection{Motivation and Background}\label{Motiv and Back}
In this article, we are interested in studying quantitatively discrete time, time-inhomogeneous Markov chains. The problem of obtaining accurate estimates for the time to reach equilibrium for a Markov chain within a time-homogeneous context is classical and has been extensively studied in the literature, see \cite{SF}. On the other hand, adapting standard techniques from the time-homogenous to the time-inhomogenous context turns out to be far from straightforward, and indeed time inhomogeneity may produce unexpected behaviors, as shown in \cite{SZ0}, \cite{SZ4} and \cite{H}.

First note that, unlike a time-homogeneous and aperiodic Markov chain, a time-inhomogenous Markov chain may not converge in law as time tends to infinity. For time-homogenous Markov chains, studying mixing properties or the convergence to equilibrium means investigating the rate of convergence of the laws towards a reference target measure. In the time-inhomogeneous context, there is generally no time-independent reference measure. However, it still makes sense to ask how long one should wait until the law at time $t$ does not depend much on the initial state. Following \cite{SZ0} and \cite{SZ3}, this observation leads us to the definition of "merging times".

Let $V$ be a finite or infinite countable set and $(K_t)_{t\in \mathbb{N}}$ a family of Markov transition operators defined on $V$. Let $\mu_0$ be a probability measure on $V$. We denote with $(\Omega\,,\,\mathcal{G}\,,\,\mathbb{P})$ a probability space and such that $(X_t^{\mu_0})_{t\in \mathbb{N}}$ is a $V$-valued discrete-time Markov chain driven by $(K_t)_{t\in \mathbb{N}}$ with initial law $\mu_0$ defined by:
$$
\mathcal{L}(X_0^{\mu_0})=\mu_0  \quad\text{and}\quad\forall t\ge 0\,,\,\forall z \in V\,,\quad \mathbb{P}(X_{t+1}^{\mu_0}= z\mid  X_t^{\mu_0})= K_{t+1}(X_t^{\mu_0},z)
\, .$$
For $t\geq 0$, let $\mu_t^{\mu_0}$ be the law of $X_t^{\mu_0}$. When $\mu_0$ is the Dirac mass at point $x$, we use the short hand notation $\mu_t^x$ instead of $\mu_t^{\mu_0}$. Furthermore, we denote for $0\leq s< t$,
$$K_{s,t}:=K_{s+1}...K_{t}\quad\text{and}\quad  K_{0,t}=K_1...K_t\,.$$

Merging corresponds to the property of forgetting the initial law of the Markov chain. To quantify this property, we control the distance between the distributions after $t$ steps of two chains driven by the same Markov transition operators and started at two distinct initial conditions. We work with the total variation distance and when $V$ is finite, it is interesting to consider the separation distance. We recall that these distances are defined as:
\begin{equation*}
    \forall (\mu,\nu)\,,\quad \, d_{TV}(\mu,\nu)= \frac{1}{2}\sum_{x\in V}|\mu(x)-\nu(x)|\quad \text{and}\quad s( \mu, \nu ) = \max_{x\in V}\left\{ 1- \frac{\mu(x)}{\nu(x)} \right\}\,.
\end{equation*}
Therefore, we aim to find bounds as precise as possible on the following quantities for $(x,y)\in V$ and $t\geq 1$:
$$d_{TV}(\mu_t^x,\mu_t^y) \quad \text{or} \quad s(\mu_t^x,\mu_t^y)\,.  $$

Imitating \cite{SZ2}, we define the following quantities:
\begin{definition}
Let $\eta\in (0,1)$ and $x,y$ in $V$. We define the $\eta$-merging time related to $x,y$ and the relative-sup $\eta$-merging time related to $x,y$ respectively by:
\begin{equation*}
	T_{\text{mer}}(x,y,\eta) = \inf\left\{ t \ge 1\mid d_{TV}(\mu_t^{x},\mu_t^{y}) \le \eta \right\}\quad \text{  and }\quad 
		T_{\text{mer}}^{\infty}(x,y,\eta) = \inf\left\{ t \ge 1\mid s(\mu_t^{x}\mid \mu_t^{y}) \le \eta\right\}\, .
\end{equation*}
\noindent
We may want to bound the merging uniformly over $x,y$. Hence, we consider the $\eta$-merging time and the relative-sup $\eta$-merging time defined by: 
\begin{equation*}
    T_{\text{mer}}(\eta) = \max\left\{ T_{\text{mer}}(x,y,\eta)\mid (x,y)\in V\right\}\quad \text{and} \quad 
    T_{\text{mer}}^{\infty}(\eta) = \max \left\{T_{\text{mer}}^{\infty}(x,y,\eta) \mid (x,y)\in V\right\}\,.
\end{equation*}
\end{definition}

Note that the merging time coincides with the classical mixing time in the time-homogeneous context up to a constant.

In order to quantify merging, we shall rely on several functional inequalities. We start discussing Poincar\'e inequalities. First, given a Markov transition operator $K$ with stationary probability $\Tilde{\pi}$, we define the following operators:
\begin{itemize}
	\item $K^*$ is the adjoint operator of $K$ from $\ell^2\left(\Tilde{\pi}\right)$ to $\ell^2\left(\Tilde{\pi}\right)$,
	\item $Q:=K^*K$ from $\ell^2\left(\Tilde{\pi}\right)$ to $\ell^2\left(\Tilde{\pi}\right)$ the multiplicative symmetrisation.
\end{itemize}
We recall the definition of the variance linked to $\Tilde{\pi}$ and the Dirichlet form linked to the pair $(Q,\Tilde{\pi})$:
$$
\begin{aligned}
\text { for all } f: V \rightarrow \mathbb{R}\,,\quad \operatorname{Var}_{\Tilde{\pi}} (f ) &= \sum_{x\in V} f^2(x)  \Tilde{\pi}(x) -\left(\sum_{x\in V} f(x)  \Tilde{\pi}(x)\right)^2 \quad \text{ and } \\
\mathcal{E}_{Q,\Tilde{\pi}}(f, f) = \sum_{x \in V}\sum_{y\in V} &(f(x)-f(y))^2 \Tilde{\pi}(x)Q(x,y)\,.
\end{aligned}
$$

The Poincar\'e constant of $Q$ is defined as follows:
\begin{definition}
Let $\gamma(Q)\geq 0$  be the optimal constant in the inequality:
$$ \quad \text { for all } f: V \rightarrow \mathbb{R} \,,\quad  \gamma \operatorname{Var}_{\Tilde{\pi}} (f )\leq \mathcal{E}_{Q,\Tilde{\pi}}(f, f)\,. $$
\end{definition}
\noindent Since $Q$ is reversible, note that $1-\gamma(Q)$ is an eigenvalue of $Q$.\\
It is then immediate to estimate the merging of a time-homogeneous Markov chain using this constant, see Corollary 2.1.5 in \cite{SF}:
\begin{theorem}\label{homogeneP}
Let $K$ be an aperiodic and irreducible Markov transition operator with invariant probability $\Tilde{\pi}$. Let $Q$ be the multiplicative symmetrisation of $K$ from $\ell^2\left(\Tilde{\pi}\right)$ to $\ell^2\left(\Tilde{\pi}\right)$ and $\gamma$ be $\gamma(Q)$, then for all $(x,y) \in V$,
$$d_{TV}\left(\mu_t^x, \mu_t^y\right) \leq \frac{1}{2}\left(\frac{1}{\sqrt{\Tilde{\pi}(x)}}+ \frac{1}{\sqrt{\Tilde{\pi}(y)}} \right)(1-\gamma)^{t/2} . $$
Therefore,
$$ \forall \eta \in (0,1)\,,\quad  T_{\text{mer}}(x,y,\eta) \leq \frac{2}{\gamma}\left(\log(1/\eta)+\log\left( \frac{1}{\sqrt{\Tilde{\pi}(x)}}+ \frac{1}{\sqrt{\Tilde{\pi}(y)}}\right)\right)\, .$$
\end{theorem}
\noindent Note that when the Markov transition operators share the same invariant probability, Theorem \ref{homogeneP} still holds, see \cite{ThomasLuca}.

A classical example of an application of Theorem \ref{homogeneP} is as follows.\\
Consider a Markov transition operator $K^N$ on $\{0, \ldots, N\}$ such that:
\begin{itemize}
    \item $K^N(x, y) \in[1 / 4,3 / 4]$ if $|x-y| \leq 1$,
    \item  $K^N$ has an invariant probability measure $\Tilde{\pi}^N$ satisfying: \\ $\forall x \in V_N\,,\,\,1 / 4 \leq(N+1) \Tilde{\pi}^N(x) \leq 4$.
\end{itemize}
Then, a comparison argument with the symmetric nearest neighbor random walk on $\{0, \ldots, N\}$ implies that $\gamma((K^{N})^* K^N)$  is of order $1/N^2$, see Example \ref{stick}.\\
Therefore, an immediate application of Theorem \ref{homogeneP} gives that there exists $\kappa >0$ a constant independent of $N$ such that:
$$ \forall \eta \in (0,1)\,,\quad  T_{\text{mer}}^N(\eta) \leq \kappa N^2\left(\log(1/\eta)+\log(N)\right)\,.$$

We now turn to the time-inhomogeneous context. For each $t\geq 1$, let $\Tilde{\pi}_t$ denote an invariant probability measure for $K_t$ and $\gamma_t$ be the Poincar\'e constant. Assume that $\gamma>0$ is a common lower bound for all the $\gamma_t$'s. One might naively expect a similar bound as in Theorem \ref{homogeneP}. However, R. Huang provides a counter-example in \cite{H}, see the following:
\begin{theorem} Let $V_N $ be $\{0, \ldots, N\}$. There exists $Q_N:=(K^N_t)_{t\ge 1}$ a sequence of nearest-neighbor Markov transition operators on $V_N$ satisfying:
    \begin{itemize}
        \item for $t\geq 1$, $K^N_t(x, y) \in[1 / 4,3 / 4]$ if $|x-y| \leq 1$,
        \item for $t\geq 1$, $K^N_t$ has an invariant probability measure  $\Tilde{\pi}^N_t$ satisfying: \\ $1 / 4 \leq(N+1) \Tilde{\pi}^N_t(x) \leq 4,$ for all $x$ in $V_N$.
    \end{itemize}
Then, for all $t \geq 1$, the Poincar\'e constant $\gamma((K_t^{N})^* K_t^N)$ is of order $\frac{1}{N^2}$  but however,
$$\underset{ N \rightarrow +\infty}{ \liminf} \frac{\log\left(T^N_{\text{mer}}(0,N,1/2)\right)}{N} >0 \,.$$ 
\end{theorem}
This counterexample illustrates the difficulty in grasping the merging time. Indeed, in this case, we assume bounds on the spectral gaps, kernels, and invariant probabilities that are uniform both with respect to time and space. Finally, to each transition operator, one can associate a set of conductances with values between $1$ and $2$. However, despite all this and contrary to intuition, the merging time is not polynomial but exponential in $N$. In conclusion, the naive extension of Theorem \ref{homogeneP} to the time-inhomogeneous context does not work. It is therefore necessary to impose hypotheses that exclude this type of case.

One approach is with $c$-stablity. In a series of papers, Saloff-Coste and Z{\'u}{\~n}iga introduce the notion of $c$-stability for time-inhomogeneous Markov chains. Under this assumption, one controls the fluctuations of the invariant probabilities. Checking $c$-stability is a challenging task, as noted by the authors. Indeed, the $c$-stability property involves the laws that are unknown and need to be estimated. Examples of time-inhomogeneous Markov chains for which it was possible to check $c$-stability are small perturbations of time-homogeneous Markov chains that furthermore satisfy some symmetries. Under the assumption of $c$-stability, a minor modification of Theorem \ref{homogeneP} holds. Its proof relies on singular values theory.

In this article, we work under a different assumption: the existence of a non-decreasing finite environment.
\begin{remark}[Notation]When $\pi$ is a finite measure, we denote by $\Tilde{\pi}$ the probability measure given by $\pi$, that is to say $\Tilde{\pi}:= \frac{\pi}{\pi(V)}$.   
\end{remark}
\begin{definition} \label{finite envir}
	Let $(K_t)_{t\ge 1}$ be a sequence of Markov transition operators and let $(\pi_t)_{t\ge 1}$ be a sequence of positive measures. We say that $\{(K_t,\pi_t)\}_{t\ge 1}$ is a non-decreasing environment when the sequence $(\pi_t)_{t\ge 1}$ satisfies the following:
    \begin{itemize}
        \item for all $ t \ge 1,\pi_tK_t =\pi_t$,
        \item for all $ t \ge 1, \pi_t(V) < +\infty$,
        \item for all $ t\ge 1,$ for all  $x\in V\,,\, \pi_{t+1}(x) \ge \pi_t(x)$.
    \end{itemize}
\end{definition}
One of the advantages of a non-decreasing finite environment is the following:
\begin{remark}\label{Inclusion}
Let $\{(K_t,\pi_t)\}_{t\ge 1}$ be a non-decreasing environment. Then, we have:
\begin{equation}\label{biendef}
	\forall t\ge 1\,,\,\forall p\in [1,+\infty]\,,\quad \ell^p(\pi_{t+1})\subset \ell^p(\pi_{t})\,.
\end{equation}  
And, for all $t\geq 1$, the operator $K_t :\ell^2(\pi_t) \to \ell^2(\pi_{t-1})$ is well-defined (we set $\pi_0=\pi_1$).
Inclusion \ref{biendef} may be a natural property to be expected. 
\end{remark}
Note that in Definition \ref{finite envir}, the $\pi_t$'s are measures but not necessarily probability measures. Indeed, if they were all probability measures then they would all be the same. \\
Moreover, when a non-decreasing finite environment exists, an infinite number of non-decreasing finite environments exists. Indeed, let $(a_t)_{t\ge 1}$ be a non-decreasing sequence of positive real numbers, if $\{(K_t,\pi_t)\}_{t\ge 1}$ is a finite non decreasing environment then,  $\{(K_t,a_t\pi_t)\}_{t\ge 1}$ is also a finite non decreasing environment.

In the examples, see Section \ref{examples}, we will leverage the availability of multiple choices for invariant measures and play with these different options to, for instance, facilitate the study of the Poincar\'e constants or to ensure that total masses are not too large.\begin{theorem}\label{PoincaréthmMer}
Let $(K_t)_{t \ge 1}$ be a sequence of irreducible aperiodic Markov transition operators and $(\pi_t)_{t\ge 1}$ a sequence of finite measures. Assume that the sequence $\{(K_t,\pi_t)\}_{t\ge 1}$ is a finite non-decreasing environment. Let $\gamma_t=\gamma(K_t^*K_t)$ be the Poincar\'e constant associated to $K_t^*K_t$. We recall that $\Tilde{\pi}_1$ denotes the probability measure $\pi_1()/\pi_1(V)$.\\
Then, for all $x,y\in V$,
\begin{equation}\label{Equation 1}
	\forall t\geq 1\,,\,d_{TV}( \mu_t^x ,\mu_t^y)\le\frac{1}{2}\sqrt{\frac{\pi_t(V)}{\pi_1(V)}} \left(\frac{1}{\sqrt{\Tilde{\pi}_1(x)}}+\frac{1}{\sqrt{\Tilde{\pi}_1(y)}}\right)\prod_{s=1}^t \sqrt{1-\gamma_s} \, .
\end{equation} 
Therefore, for all $\eta\in(0,1)$,
\begin{equation}\label{Estime2}
	T_{\text{mer}}(x,y,\eta) \le \min\left\{t\ge 1\mid  \frac{1}{2}\sqrt{\frac{\pi_t(V)}{\pi_1(V)}}\left(\frac{1}{\sqrt{\Tilde{\pi}_1(x)}}+\frac{1}{\sqrt{\Tilde{\pi}_1(y)}}\right)\prod_{s=1}^t \sqrt{1-\gamma_s}  \le \eta\right \}\, .
\end{equation} 
\end{theorem}
\begin{remark}
    We adopt the convention:
    $$ \min \emptyset = +\infty \,. $$
    Thus, when the minimum in \ref{Estime2} in Theorem \ref{PoincaréthmMer} is infinite, then, we get $T_{\text{mer}}(x,y,\eta) \leq +\infty$. We adopt the same convention in Theorems \ref{MergingNash},\ref{LogSobolevThm} and \ref{LogSobolevThm2}.
\end{remark}
\begin{remark}
Note that Estimate \ref{Equation 1} is very close to the conclusions of Theorem \ref{homogeneP}. In addition, when $(\pi_t)_{t\ge 1}$ and $(K_t)_{t\ge 1}$ are constant, we retrieve the exact statement of Theorem \ref{homogeneP}.    
\end{remark}
To ensure the assumption of a finite non-decreasing environment, the challenge lies in computing explicit invariant probability measures and our ability to compare them. Roughly speaking, there is a competition between the ratio of the total masses and the product of eigenvalues to achieve a reasonably small merging time. When the ratio of the total masses is too large, the right-hand side of Equation \ref{Equation 1} might not tend to $0$, see Example \ref{unbounded case 1}. In fact, the existence of a non-decreasing environment does not imply that merging will occur, see Example \ref{unbounded case 1}. Our Estimate \ref{Estime2} is not always sharp, see Example \ref{stick}.

If $V$ is a finite set and $(K_t)_{t\ge 1}$ a family of irreducible Markov transition operators, there always exists a family of finite measures $(\pi_t)_{t\ge 1}$ such that $\{(K_t,\pi_t)\}_{t\ge 1}$ is a finite non-decreasing environment. However, the ratio of the total masses can increase very fast and this deteriorates our estimate.

A class of examples where the hypothesis on the existence of a finite non-decreasing environment is easy to check is Markov chains on electrical networks, see \cite{LPeres} for further details. Then, the sum of the conductances at a given vertex provides an explicit invariant measure. We give a definition/proposition:
\begin{definition}\label{ElectricNetwork}
    Let $E$ be a subset of $V\times V$ such that if $(x,y) \in E$, then, $(y,x) \in E$. We name $E$ a set of edges. A set of conductances on $E$ is a function $c:E \mapsto \mathbb{R}$ satisfying:
    $$ \forall(x,y)\in E\,,\quad c(x,y)\geq 0 \quad \text{and}\quad c(x,y)=c(y,x)\,.$$
    We associate to the set of conductances $c$ the Markov transition operator $P$ defined by:
    $$\forall (x,y)\in V\,,\quad P(x,y)=\frac{c(x,y)}{\sum_{z\in V}c(x,z)}\, .$$
     We add the following assumption:
    \begin{equation}\label{mesurefinie}
     \sum_{x\in V}\sum_{y\in V} c(x,y) <+\infty\, .  
    \end{equation}
    Let $\pi$ be the positive measure defined by:
    \begin{equation}\label{egalite1}
    \forall x \in V\,, \quad \pi(x)=\sum_{y\in V}c(x,y)\, .
    \end{equation}
    Assumption (\ref{mesurefinie}) implies that $\pi$ is a finite measure.
    Note that $P$ is reversible with respect to $\pi$. Indeed,
    $$\forall (x,y)\in V\,, \quad \pi(x) P(x,y) = c(x,y)=c(y,x)=\pi(y)P(y,x) \,.$$
    It implies that $\pi$ is a finite invariant distribution for $P$. The pair \((P, \pi)\) induced by \(c\) is called an electrical network.
\end{definition}
\begin{remark}\label{reciproquedefonduct}
When the Markov transition operator $P$ is self-adjoint from $\ell^2$ to $\ell^2$ with respect to some measure then, there exists a set of conductances $c$ so that:
$$\forall (x,y)\in V\,,\quad  P(x,y)=\frac{c(x,y)}{\sum_{z\in V}c(x,z)}\, .$$
Note that multiple choices of set of conductances can be associated with $P$. More precisely, if $c$ is a set of conductances associated with $P$, then for any $\lambda\in \mathbb{R}_+^*$, the set of conductances $\lambda c$ is also associated with $P$.
\end{remark}
A sequence of electrical networks with non-decreasing conductances forms a non-decreasing environment. A typical example of a time-inhomogeneous Markov with a non-increasing environment is as follows: we fix a set of edges $E$, a reference set of conductances $c$ and a positive constant $M>0$. Consider a sequence of sets of conductances sets $(c_t)_{t\ge 1}$ such that:
$$ \forall t\geq 1\,,\,\forall (x,y) \in E\,,\quad  c(x,y) \leq c_t(x,y) \leq Mc(x,y) \,. $$
For $t\geq 1$, let $P_t$ be the Markov transition operator associated with $c_t$, see Definition \ref{ElectricNetwork}. At time $t-1$, we apply the operator $K_t:=\frac{1}{2}(I+P_t)$. We will see that this type of example has several advantages, see Example \ref{stick}. The first one is that the sequence $(\pi_t)_{t\ge 1}$ given by (\ref{egalite1}) gives a finite non-decreasing environment. The second advantage is, for example, that if the operator associated with $c$ satisfies a Poincar\'e condition, then \( K_t \) does as well, with a Poincar\'e constant comparable to the original one.

One method for generalizing electrical networks is to consider graphs with cycles, see \cite{Cycles} and \cite{Math} for a presentation of cycles.

Additionally, note that we may construct examples of sequences of electrical networks where the $\pi_t$'s are non-decreasing but the conductances themselves are not monotonous, see Section \ref{fromtoanother}.

Effective tools for studying $\ell^2$-merging times and complementary to spectral techniques are advanced functional inequalities: Nash inequalities and logarithmic Sobolev inequalities. We refer to the lecture notes from St. Flour by Saloff-Coste \cite{SF} for details on the use of these functional inequalities in the context of time-homogeneous Markov chains with a finite state space.
Saloff-Coste and Z\'u\~niga successfully applied these functional inequalities under the $c$-stability assumptions to estimate the merging time of a time-inhomogeneous chain.
Now, we show the results obtained using Nash and logarithmic Sobolev inequalities. First, we need to define a relative point-wise distance in relation to a probability measure.
\begin{definition}
Let $V$ be a finite set. Fix $\Tilde{\pi}$ a probability measure on $V$. Assume that for all $x\in V$, $\tilde{\pi}(x)>0$. We define the relative point-wise distance in relation to $\Tilde{\pi}$ by:
\begin{equation*}
   \text{for all probability measures } (\mu, \nu) \text{ on } V\,, \quad s_{\infty}( \mu \,,\,  \nu\mid \tilde{\pi} ) = \max\left\{ \left| \frac{\mu(x)-\nu(x)}{\Tilde{\pi}(x)}\right| \mid x\in V \right\}\,.
\end{equation*}
\end{definition}
\begin{definition}\label{dynamicmergingtime}
Let $\{(K_t,\pi_t)\}_{t\ge 1}$ be a finite non-decreasing environment.
We define a dynamic merging-time in relation to the sequence $(\pi_t)_{t\geq 1}$ by:
$$\forall \eta \in (0,1)\,,\quad T_{\text{mer}}^{\infty}( \eta \,, (\pi_t)_{t\geq 1} ) = \inf \left\{ t\geq 1 \mid \forall (x,y) \in V\,,\, s_{\infty}( \mu_t^x \,,\,  \mu_t^y\mid \tilde{\pi}_t ) \leq \eta\right\}\,.$$
\end{definition}
\begin{lemma}\label{lemmetempsmelange}
Let $\{(K_t,\pi_t)\}_{t\ge 1}$ be a finite non-decreasing environment. For each $t\geq 1$, we recall that $\tilde{\pi}_t$ denotes $\pi_t/\pi_t(V)$. For $(x,y)$ in $V$, the sequence $\left(\frac{1}{\pi_t(V)} s_{\infty}( \mu_t^x \,,\,  \mu_t^y\mid \tilde{\pi}_t) \right)_{t\geq 1}$ is non-increasing.
\end{lemma}
We recall that a Nash inequality is defined as follows:
\begin{definition}\label{DefNashIn}
Let \( T \geq 1 \), \( Q \) be a Markov transition operator, and \( \Tilde{\pi} \) a probability measure. Let \( C, D > 0 \). We say that the pair \( (Q, \Tilde{\pi}) \) satisfies the \( \mathcal{N}(C, D, T) \) Nash inequality if the following holds:
$$
\forall f: V \rightarrow \mathbb{R}\,, \quad\|f\|_{\ell^2(\Tilde{\pi})}^{2+1 / D} \leq C\left( \mathcal{E}_{Q,\Tilde{\pi}}(f,f) 	 +\frac{1}{T}\|f\|_{\ell^2\left(\Tilde{\pi}\right)}^2\right)\|f\|_{\ell^1(\Tilde{\pi})}^{1 / D} \,.
$$
\end{definition}

Under this assumption, we have:
\begin{theorem}\label{MergingNash}
Let $(K_t)_{t \ge 1}$ be a sequence of irreducible aperiodic Markov transition operators on a finite set $V$ and $(\pi_t)_{t\ge 1}$ a sequence of positive measures. 
Assume that the sequence $\{(K_t,\pi_t)\}_{t\ge 1}$ is a finite non-decreasing environment. For $t\geq 1$, let $Q_t=K_t^*K_t$ and $\gamma_t$ be $\gamma(Q_t)$. Let $T \geq 1$ and $C, D>0$. Assume that for all $1\leq t\leq T$, the pair $(Q_t,\Tilde{\pi}_t)$  satisfies the $\mathcal{N}(C,D,T)$ Nash inequality. Assume $\pi_1(V)\ge 1$ and define $B$ by: $$B=B(D, T)=2^{1/D}(1+1 / T)(1+[ 4 D])\,.$$
Then, for all $x,y\in V$,
\begin{equation}
	\forall r\leq T\,,\,\forall t\geq r\,,\quad d_{TV}(\mu_t^x,\mu_t^y)\le 2\sqrt{\pi_t(V)}\left(\frac{4 C B}{r+1}\right)^D\prod_{l=r+1}^t \sqrt{1-\gamma_l} \, .
\end{equation} 
And, for all $\eta \in (0,1)$,
\begin{equation*}
T_{\text{mer}}(\eta) \le \min\left\{t\ge 0\mid  \exists\,   r \leq T \,,\, t-r\geq 0 \quad  \text{and}\quad  2\sqrt{\pi_t(V)}\left(\frac{4 C B}{r+1}\right)^D\prod_{l=r+1}^t \sqrt{1-\gamma_l}  \le \eta\right\}\,.
\end{equation*}
Moreover, we have for $r \le T$ and $t= 2r+u,$
\begin{equation}
	\max \left\{ s_{\infty}(\mu_t^x,\mu_t^y \mid \Tilde{\pi}_t) \mid (x,y)\in V\right\} \le  8\pi_t(V)\left(\frac{4 C B}{r+1}\right)^{2D}  \prod_{l=r+1}^{t-r} \sqrt{1-\gamma_l}\, .
\end{equation} 
So, for all $\eta \in (0,1)$,
$$ T_{\text{mer}}^{\infty}( \eta \,, (\pi_t)_{t\geq 1} ) \leq  \inf \left\{ t\geq 1 \mid  8\pi_t(V)\left(\frac{4 C B}{r+1}\right)^{2D}  \prod_{l=r+1}^{t-r} \sqrt{1-\gamma_l} \leq \eta\right\}\,. $$
\end{theorem}
\begin{remark}
    The form of Nash inequality in Definition \ref{DefNashIn} does not imply that the Poincar\'e constant is not 1. The advantages of the Nash inequality in Definition \ref{DefNashIn} are explained in the context of time-homogeneous Markov chains in Section 2.2.3 of \cite{SF}.
\end{remark}
We state a bound on the merging time in finite non-decreasing environments that satisfy a logarithmic Sobolev inequality at each time. We recall the expression of the entropy:
\begin{definition}\label{definitionentropy}
Let $\Tilde{\pi}$ be a probability measure in $V$. Let $f$ be a non-negative function in $\ell^1(\Tilde{\pi})$. We define the entropy of $f$ with respect to $\Tilde{\pi}$ by:
$$\mathcal{L}(f\mid \Tilde{\pi}):= \sum_{x\in V} f(x)\log\left(\frac{f(x)}{\Tilde{\pi}(f)}\right)\Tilde{\pi}(x)\,.$$
\end{definition}
\noindent
Note that Jensen inequality implies that $\mathcal{L}(f|\Tilde{\pi})$ is non-negative.

The definition of the logarithmic Sobolev constant is similar to the one of the Poincar\'e constant except that the variance is replaced by the entropy. 
\begin{definition}\label{definitionconstantelogso}
The standard logarithmic Sobolev constant of $Q$ is the constant\\ $\alpha(Q)\geq 0$ defined as the optimal constant in the inequality: $$\forall f: V \rightarrow \mathbb{R} \,,\quad f- \text{non constant and non-negative}\,,\quad\frac{\mathcal{E}_{Q,\Tilde{\pi}}(f,f)}{\mathcal{L}(f^2
\mid \Tilde{\pi})} \ge \alpha\, . $$
\end{definition}

Thanks to hypercontractivity, we can control the merging times for finite non-decreasing environments with the following:\begin{theorem}\label{LogSobolevThm}
Let $(K_t)_{t \ge 1}$ be a sequence of irreducible aperiodic Markov transitions operators and $(\pi_t)_{t\ge 1}$ a sequence of positive measures. Assume that the sequence $\{(K_t,\pi_t)\}_{t\ge 1}$ is a finite non-decreasing environment. For each $t\geq 1$, let $\gamma_t=\gamma(K_t^*K_t)$ be the Poincar\'e constant associated to $K_t^*K_t$ and $\alpha_t$ be the logarithmic Sobolev constant $\alpha(K_t^*K_t)$. For all $s\geq1$ and $z\in V$, we define:
$$q_s=2\prod_{u=1}^s (1+\alpha_u)\quad \text{and}\quad r_z=\min\left\{ s \ge 1\mid   \log(q_s) \ge \log\left(\log\left(\frac{1}{\Tilde{\pi}_1(z)}\right)\right)\right\} \, .$$
For $x,y \in V$, let $r(x,y) =\max(r_x,r_y)$. Then, 
\begin{equation}
    \forall t \ge r(x,y)\,, \quad d_{TV}(\mu_t^x,\mu_t^y)\leq 
    e\sqrt{\frac{\pi_t(V)}{\pi_1(V)}}\prod_{l=r+1}^t \sqrt{1-\gamma_l} \, .
\end{equation} 
Therefore, for all $\eta\in (0,1)$,
\begin{equation}
	T_{\text{mer}}(x,y,\eta) \le r(x,y)+\min\left\{u\ge 1\mid e\sqrt{\frac{\pi_t(V)}{\pi_1(V)}}\prod_{l=r+1}^{r+u} \sqrt{1-\gamma_l} \le \eta\right\}\, .
\end{equation}\end{theorem}
\begin{lemma}\label{comparaisga}
    Let $Q$ be a Markov transition operator. Assume that $Q$ is aperiodic then, the logarithmic Sobolev constant $\alpha(Q)$ is non-negative. And, the logarithmic Sobolev constant $\alpha(Q)$ in Definition \ref{definitionconstantelogso} and the Poincar\'e constant $\gamma(Q)$ satisfy: $$2 \alpha(Q)\leq 1-\gamma(Q) \,.$$
    When the logarithmic Sobolev constant is strictly positive, the Poincar\'e constant is not 1.
\end{lemma}
When the state set is finite and when the logarithmic Sobolev constants are uniformly bounded from below, one can control the relative-sup merging time:
\begin{theorem}\label{LogSobolevThm2}
Let $(K_t)_{t \ge 1}$ be a sequence of irreducible aperiodic Markov transition operators and $(\pi_t)_{t\ge 1}$ a sequence of positive measures. Assume that the sequence $\{(K_t,\pi_t)\}_{t\ge 1}$ is a finite non-decreasing environment. Assume that $V$ is finite. Let $\gamma_t=\gamma(K_t^*K_t)$ be the Poincar\'e constant associated to $K_t^*K_t$, let $\alpha_t$ and $\hat{\alpha}_t$ be the logarithmic Sobolev constant $\alpha(K_t^*K_t)$ and $\alpha(K_tK_t^*)$ respectively. Assume that for all $t\geq 1$, the constants $\alpha_t$ and $\hat{\alpha}_t$ are bounded from below by $\alpha >0$. Let $\Tilde{\pi}_t^{\#}=\min\left\{\tilde{\pi}_t(x)\mid x\in V\right\}$. And, assume that it exists $\rho >0$ such that:
$$ \forall t \ge 1\,,\quad \tilde{\pi}_t^{\#} \ge \rho \, .$$ 
Define $r$ by:
$$r= \left( \frac{\log\left(\frac{\log(1/\rho)}{2}\right)}{\log(1+\alpha)}\right) +1\,.$$
Let $u\geq 0$ and set $t =2r+u$. Then, we find:
$$	\max \left\{ s_{\infty}(\mu_t^x,\mu_t^y \mid \Tilde{\pi}_t) \mid (x,y)\in V\right\} \le 4e^2\sqrt{\frac{\pi_{t-r}(V)}{\pi_1(V)} \prod_{l=r+1}^{r+u} 1-\gamma_{r}}\,. $$
So, for all $\eta \in (0,1)$,
$$ T_{\text{mer}}^{\infty}( \eta \,, (\pi_t)_{t\geq 1} ) \leq \inf \left\{   u +2r \mid u\geq 1\,,\,  4e^2\sqrt{\frac{\pi_{r+u}(V)}{\pi_1(V)} \prod_{l=r+1}^{r+u} 1-\gamma_{r}}\leq \eta \right\}\,.$$
\end{theorem}
\begin{remark}
    The concept of monotonous invariant measures already appeared in the literature to deal with issues of recurrence versus transience, see \cite{AM}.\\
    In \cite{DAHZ}, the idea of non-decreasing invariant measure is highlighted. Dembo, Huang, and Zheng establish two-sided Gaussian estimates for the transition kernel of a time-inhomogeneous Markov chain on a non-decreasing sequence of electrical networks satisfying a suitable uniform Poincar\'e inequality and the volume doubling property.\\
    In our paper, in the proofs, we have succeeded in adapting ideas of \cite{DAHZ} such as backward induction in the Poincar\'e case. However, unlike in \cite{DAHZ}, some extra care is needed because we have to center our test functions with respect to measures that evolve over time.
\end{remark}
\subsection{Outline}
This article is divided into three parts. Section \ref{Introdu} presents the motivations, background, and main theorems. In Definition \ref{finite envir}, we introduce our main hypothesis: the existence of a non-decreasing finite environment and a new merging time related to a sequence of measures, see Definition \ref{dynamicmergingtime}. Each theorem extends classical functional inequalities to the non-decreasing finite environment hypothesis. The first theorem, Theorem \ref{PoincaréthmMer}, assumes a Poincar\'e condition at each time, in addition to the non-decreasing finite environment hypothesis, and states control of the merging time in total variation distance. The second theorem, Theorem \ref{MergingNash}, assumes a Nash inequality at each instance and states control of the merging time in total variation distance and yields a control on the dynamic merging-time. Finally, Theorem \ref{LogSobolevThm} assumes a Logarithmic Sobolev inequality at each time and similarly states control of the merging time in total variation distance and yields a control on the dynamic merging-time when the set $V$ is finite with Theorem \ref{LogSobolevThm2}. Theorems \ref{PoincaréthmMer},\ref{MergingNash} and \ref{LogSobolevThm2} are new tools for understanding finite Markov chains in the inhomogeneous time context.

Section \ref{examples} deals with examples and investigates five instances within the context of electrical networks. The first example serves an educational purpose, while the subsequent two illustrate the application of Theorem \ref{PoincaréthmMer}. The final two examples are within the classical framework for the application of Nash inequalities and Logarithmic Sobolev inequalities.

The third section delves into the proofs and is divided into four sub-sections. The first one revisits the preliminaries and classical definitions from analysis and probability theory. The second part focuses on proving Theorem \ref{Poincaréthm} and deals with the Poincar\'e constants. It includes a crucial proposition regarding the duality of operators within the inhomogeneous time context. The third segment contains the proof of Theorem \ref{MergingNash}, reiterating the definition of a Nash inequality. Finally, in the fourth subsection, we adopt an entropic approach and prove Theorem \ref{LogSobolevThm} and \ref{LogSobolevThm2}.

\newpage
\section{Examples}\label{examples}
In this section, we discuss pedagogical examples and other interesting cases.
\begin{example}[Conductances on a stick]\label{stick}  
Let $N\geq 1$ and $V_N$ be $\mathbb{Z}/N\mathbb{Z}$ and let the set of edges $E_N$ be $\{ (x,y) \in \mathbb{Z}/N\mathbb{Z}\mid  |x-y|=1\}$. Let $c_0$ be the set of conductances defined by:
$$ \forall (x,y)\in E_N\,,\quad c_0(x,y) =1 \,. $$
Let $(c_t)_{t\ge 1}$ be a sequence of sets of conductances on $E_N$. Assume that:
\begin{itemize} \item for all $ t \geq 1,\, c_t \geq c_0$,
\item and for all $ (x, y) \in E_N$, the sequence $(c_t(x,y))_{t \geq 1}$ is non-decreasing. \end{itemize}
\begin{figure}[H]
    \centering
    \includegraphics[scale=0.4]{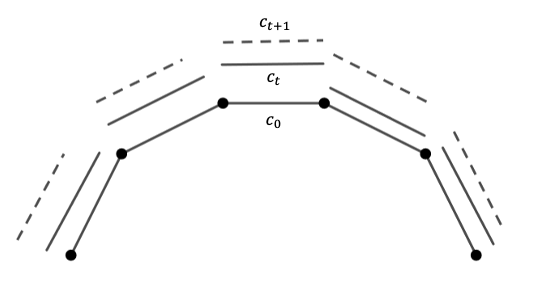}
    \caption{Non-decreasing sets of conductances}
    \label{fig: croistick}
\end{figure}
\noindent
For $t \geq 1$, let $P_t$ be the Markov transition operator on $V_N$ associated to the set of conductances $c_t$ and let $\pi_t$ be its reversible measure, see Definition \ref{ElectricNetwork}. We study the Markov transition operators $(K_t)_{t\ge 1}$ defined by:
$$\forall t\ge 1\,,\quad K_t =\frac{1}{2}(I+P_t)\,. $$
For each $t\geq 1$, the operators $P_t$ and $K_t$ are reversible with respect to $\pi_t$. Let $\gamma_t$ be the Poincar\'e constant $\gamma(K^*_tK_t)$. We study the $K_t$'s instead of the $P_t$'s to compare Dirichlet forms. \\~\\
\textbf{We give a lower bound on the Poincar\'e constants under the following assumption: }
\begin{equation}
    \exists M >0 \,,\forall t \ge 1, \forall x \in V_N \,,\quad \sum_{y\in V_N} c_t(x,y) \leq M \,.
\end{equation}
Let \( P \) denote the Markov transition operator associated with the set of conductances \( c_0 \), and let \( u \) be its reversible measure. The operator \( P \) corresponds to the Markov transition operator of the simple random walk. Denote with $\gamma_u$ the constant $\gamma(P)$. It is well known, see \cite{Pbook} that:
$$\exists \kappa >0 \,,\quad \gamma_u =\frac{\kappa}{N^2}\,.$$
We define for $\pi$ a finite measure on $V_N$:
\begin{equation}\label{defintionvariancenonproba}
    \forall f :V_N \rightarrow \mathbb{R}\,,\quad \text{Var}_{\pi}(f):= \pi(V_N)\text{Var}_{\Tilde{\pi}}(f)\,.
\end{equation}
Note that:
\[
\gamma_u = \inf_{\substack{f \in \ell^2(u) \\ f \text{ non-constant}}} \left\{ \frac{\mathcal{E}_{P,u}(f)}{\text{Var}_{u}(f)}\right\} \quad \text{and} \quad \forall t \geq 1\,,\quad  \gamma_t = \inf_{\substack{f \in \ell^2(\Tilde{\pi}_t) \\ f \text{ non-constant}}}\left\{  \frac{\mathcal{E}_{K_t^*K_t,\pi_t}(f)}{\text{Var}_{\pi_t}(f)} \right\}\,.
\]
We use a comparison method in order to estimate the spectral gap of $K^*_tK_t$. We get:
$$
\begin{aligned}
\forall f \in \ell^2(\Tilde{\pi}_t)\,,\, \mathcal{E}_{P,u}(f, f)&=\frac{1}{2} \sum_{x,y}(f(x)-f(y))^2 u(x) P(x, y)\\
&=\frac{1}{2} \sum_{x,y\,,\, |x-y|=1}(f(x)-f(y))^2  \\
& \leq \frac{1}{2} \sum_{x \neq y}(f(x)-f(y))^2 c_t(x,y)\quad\text{since }c_t(x,y) \geq \mathbf{1}_{|x-y|=1}\,, \\
& \leq \frac{1}{2} \sum_{x \neq y}(f(x)-f(y))^2\pi_t(x) P_t(x,y) \\
& = \mathcal{E}_{P_t,\pi_t}(f, f) \leq 2\mathcal{E}_{K_t,\pi_t}(f, f)\,.
\end{aligned}
$$
Note that $\forall x \in V_N\,,\, \forall t\geq 1\,,\, K_t^*(x, x)= K_t(x, x) \geq 1/2$, and then, we get:
$$
\begin{aligned}
\forall y \in V\,,\quad \quad \pi_t(x) K_t^*K_t(x, y) & \geq \pi_t(x) K_t^*(x, x) K_t(x, y)+\pi_t(x) K_t^*(x, y) K_t(y, y)\,, \\
& \geq \frac{1}{2}\left( \pi_t(y) K_t(y, x)+ \pi_t(x) K_t^*(x, y)\right)\,.
\end{aligned}
$$
We finally find:
$$
\mathcal{E}_{K_t^*K_t,\pi_t}(f, f) \geq \mathcal{E}_{K_t,\pi_t}(f, f)\geq \frac{1}{2} \mathcal{E}_{P_t,\pi_t}(f, f) \,.
$$
Next, we compare variances:
$$
\begin{aligned}
\forall f \in \ell^2(\Tilde{\pi}_t)\,,\quad \operatorname{Var}_{\pi_t}(f)&= \pi_t(V_N) \operatorname{Var}_{\tilde{\pi}_t}(f) \\ 
&= \pi_t(V_N) \sum_{x\in V_N} \left(f(x)- \tilde{\pi}_t(f)\right)^2 \tilde{\pi}_t(x) \\
&\leq \pi_t(V_N) \sum_{x\in V_N} \left(f(x)- \tilde{u}(f)\right)^2 \tilde{\pi}_t(x) \\
&= \sum_{x\in V_N} \left(f(x)- \tilde{u}(f)\right)^2 \pi_t(x)  \\
&\leq  M \sum_{x\in V_N} \left(f(x)- \tilde{u}(f)\right)^2 u(x) \quad \text{using that }\, \pi_t \le M u\,, \\ 
&= M \operatorname{Var}_{u}(f) \,.    
\end{aligned}
$$
We finally find that:
\begin{equation*}
    \forall t \ge 1\,,\quad
\gamma_t \geq  \frac{1}{2M}\gamma_u \,.
\end{equation*}
\textbf{We now give a bound on the merging time .}\\~\\
Let \( T_{\text{mer}} \) be the merging time of the time-inhomogeneous Markov chain driven by the sequence $(K_t)_{t\ge 1}$. Let us summarize for all $x\in V_N$, the sequence $(\pi_t(x))$ is non-decreasing and:
$$
\forall t\geq 1\,,\,\, \gamma_t \geq \frac{\kappa}{2MN^2}\quad ,\,\forall x \in V_N\,,\quad 1 \leq \pi_t(x) \leq M u(x)\,.
$$
Theorem \ref{PoincaréthmMer} yields that there exists a constant $A$ independent of $N,M$ such that:
$$\forall \eta \in (0,1)\,,\quad T_{\text{mer}}(\eta)\le  A M N^2\left(\log(N)+\log(M)+\log(1/\eta)\right) \,.$$
\end{example}
\begin{remark}
Let $\hat{K}$ be a Markov transition operator on $V_N$ associated with a set of conductances $\hat{c}$ on $E_N$ satisfying:
$$  \forall (x,y) \in E_n\,,\quad  c_0(x,y) \leq \hat{c}(x,y) \leq M/2 \,. $$
Denote $\hat{\pi}$ the invariant probability of $\hat{K}$ and let $\hat{\mu}_t^{x}$ be the law of the chain at time $t$ started at $x$ and driven by $\hat{K}^t$. Theorem \ref{homogeneP} gives that there exists a constant $B$ independent of $N, M$ such that:
$$\forall \eta \in (0,1)\,,\quad T_{\text{mer}}(\eta)\le B M N^2\left(\log(N)+\log(1/\eta)\right) \,.$$
Thus, we obtain an estimate of the merging time similar to the one in the time-homogeneous case. The only difference is an extra $\log(M)$. Furthermore, an argument based on a Nash inequality in the time-homogeneous case allows us to refine this bound. We will do so for this example in the time-inhomogeneous case, see Example \ref{exampleboite}.
\end{remark}
\begin{example}[Example \ref{stick} with unbounded conductances]\label{unbounded case 1}
We keep the same notation as in Example \ref{stick}. We still assume that:
\begin{itemize} \item for all $ t \geq 1,\, c_t \geq c_0$,
\item and for all $ (x, y) \in E_N$, the sequence $(c_t(x,y))_{t \geq 1}$ is non-decreasing. \end{itemize}
For $t\geq 1$, let $\phi(t)=\underset{x\in V_N}{\sup}\{\pi_t(x)\}$. We consider the case:
\begin{equation}\label{condinfini}
\underset{t \ge 1}{\sup}\{ \phi(t)\} =  +\infty  \,. 
\end{equation}
Fix $\eta \in (0,1)$. Imitating the calculations of Example \ref{stick}, we find that there exists $\hat{\kappa}>0$ so that if $t\geq 1$ satisfies: \begin{equation}\label{to solve}
t \geq  \hat{\kappa} \phi(t) N^2\left(\log(\phi(t))+\log(N)+\log(1/\eta)\right)\,.
\end{equation}
Then, $$ T_{\text{mer}}(\eta) \leq t \,.$$
We distinguish two cases on $\phi$. First, we assume that:
$$\forall t\geq 1\,,\quad \phi(t)\geq t\,.$$
then Equation (\ref{to solve}) has no solution and in this scenario, merging may not occur, see Example \ref{exmplnomerg}. Now, let $\alpha \in (0,1)$ and assume that:
$$\forall t\geq 1\,,\quad \phi(t)=t^{\alpha}\,.$$ 
Here, despite Condition (\ref{condinfini}), we can control the merging, Equation (\ref{to solve}) becomes:
$$
t^{(1-\alpha)} \ge \hat{\kappa} N^2\left(\alpha\log(t)+\log(N)+\log(1/\eta)\right)\,.
$$
Then, it exists a constant $A:=A(\alpha)$ independent of $N$ such that: 
\begin{equation*}
     T_{\text{mer}}(\eta) \leq A\left( N^2\left((\alpha+1)\log(N)+\log(1/\eta)\right)\right)^{1/(1-\alpha)}\,.
\end{equation*}
We obtain a polynomial bound on the merging time.
\end{example}
\begin{question}
Can one build an example on $V_N$ with non-decreasing sets of conductances where $\phi(t)=t^{\alpha}$ and prove that the bound $ T_{\text{mer}}(\eta) = B(\alpha) N^{2/(1-\alpha)}(1+\log(1/\eta))$ is sharp?
\end{question}
\begin{example}\label{exmplnomerg}
Let $V =\{0\,,\,1\}$ with 
$\{\{1\,,\,0\}\,,\,\{0\,,\,1\}\,,\,\{1\,,\,1\}\,,\,\{0\,,\,0\}\}$ as edges. Fix $\epsilon>0$ and define the set of conductances $(c_t)_{t\geq 1}$ by:
$$
\forall t\ge 1\,, \quad c_t(0,0)=c_t(1,1)= t^{1+\epsilon} \quad\text{and}\quad c_t(1,0)=c_t(0,1)=1\,.
$$
For $t \geq 1$, let $P_t$ be the Markov transition operator on $V$ associated to the set of conductances $c_t$ and let $\pi_t$ be its reversible measure, see Definition \ref{ElectricNetwork}. We have:
$$
\forall t\ge 1, \quad P_t(0,0)=P_t(1,1)= \frac{t^{1+\epsilon}}{t^{1+\epsilon}+1} \quad \text{and}\quad P_t(1,0)=P_t(0,1)=\frac{1}{t^{1+\epsilon}+1}\,.
$$
And,
$$\forall t\ge 1, \quad \pi_t(1)= \pi_t(0)=  t^{1+\epsilon}+1 \,.$$
Note that $\{(P_t,\pi_t)\}_{t\ge 1}$ is a finite non decreasing environment. And, the sequence $(\Tilde{\pi}_t)_{t\geq 1}$ is constant equal to the uniform probability distribution over \(V\).\\
For $x\in V$, denote $(X_t^x)_{t\ge 1}$ the  Markov process $V$-valued driven by $(P_t)_{t\ge 1}$ defined by:
$$
X_0^x =x \quad \text{and}\quad \forall t\ge 0\,,\,\forall z \in V\,, \quad \mathbb{P}(X_{t+1}^x= z\mid  X_t^x)= P_{t+1}(X_t^x,z)\,.
$$
Remark that for all $x\in V$, we have:
$$
 \forall t\geq 1\,,\quad \log\left(\mathbb{P}(\forall s \in [[0,t]]\,,\, X_{s}^x= x \mid  X_0^x=x)\right)= \sum_{s=1}^{t} \log\left(1-\frac{1}{s^{1+\epsilon}+1}\right)\,. 
$$
Thus, we get:
$$
\forall x \in V\,,\quad \mathbb{P}(\forall s \geq 0\,,\, X_{s}^x= x\mid  X_0^x=x) >0 \,.
$$
Finally, there is no merging.
\end{example}

We give an example with directed cycles. For more details on Markov chains on graphs with cycles, we refer to \cite{Cycles} and \cite{Math}. It can not be reduced to an example with conductances.
\begin{example}[Cycles in $\mathbb{Z}^2$]\label{cyclesexamples}
Fix $N\geq 1$, let $V_N^2$ be $(\mathbb{Z}/N\mathbb{Z})^2$ and $E_N=\left\{(x,y) \mid |x-y|\leq 1\right\} $ be the set of edges. We define a set of conductances $c_0$ by:
$$ \forall (x,y)\in E_N\,,\quad c_0(x,y) = 1_{x\ne y} \,. $$
We recall that a directed cycle in the graph \( (V_N^2, E_N) \) of size \( k \geq 2\) is a finite sequence of vertices \( (v_1, v_2, \dots, v_k) \) such that:
\begin{itemize}
    \item each pair \( (v_i, v_{i+1}) \) is an edge in \( E_N \), with the convention \( v_{k+1} = v_1 \),
    \item the vertices $v_1,\dots v_{k}$ are distinct. 
\end{itemize}
Let \( C_4 \) be the set of all possible directed cycles of size 4 in \( (V_N^2, E_N) \). Note that any vertex in \( V_N^2 \) is part of 8 distinct directed cycles in $C_4$. Assume that a sequence of sets of weights \( (s_t)_{t\geq 1} := ((s_t(\sigma))_{\sigma \in C_4})_{t\geq 1} \) is given such that  for  any $ \sigma \in C_4\,$:
\begin{itemize} \item for all $ t \geq 1,$ the quantity $s_t(\sigma)$ is non-negative,
\item the sequence $(s_t(\sigma))_{t \geq 1}$ is non-decreasing. \end{itemize}
We also assume that there exists $M>0$ such that:
$$ \forall t\geq 1\,,\, \forall x\in V_N^2\,, \quad  \sum_{\sigma \in C_4\,,\, x\in\sigma} s_t(\sigma) \le M < +\infty \,.$$
For \( t \geq 1 \), we define \( K_t \) as a Markov transition operator for \( (x,y) \in V_N^2 \) by:
\[
K_t(x,y) =  
\frac{c_0(x,y) + \sum_{\sigma \in C_4\,,\, (x,y) \in \sigma} s_t(\sigma)}{2\left( \sum_{z\in V_N^2} c_0(x,z) + \sum_{\sigma \in C_4\,,\, (x,z) \in \sigma} s_t(\sigma)\right)  } \, \text{ if } x \ne y\quad \text{and}\quad K_t(x,x)=\frac{1}{2}\,.
\]
Here $(x,y)$ is a directed edge. We define a positive measure $\pi_t$ by:
$$\forall x\in V_N^2\,,\quad \pi_t(x)=  \sum_{z\in V_N^2} c_0(x,z) + \sum_{\sigma \in C_4\,,\, (x,z) \in \sigma} s_t(\sigma)\,. $$
Note that \( \pi_t \) is an invariant measure for \( K_t \). Let $P$ be the Markov transition operator associated to the set of conductances $c_0$, see Definition \ref{ElectricNetwork} and denote by \( u \) the measure that is constantly equal to 1 on \( V_N^2 \). We recall that $uP=u$. We define the following Poincaré constants:
\[
\gamma_u = \inf_{\substack{f \in \ell^2(u) \\ f \text{ non-constant}}} \left\{ \frac{\mathcal{E}_{P,\tilde{u}(f)}}{\text{Var}_{\tilde{u}}(f)}\right\} \quad \text{and} \quad \forall t \geq 1\,,\quad  \gamma_t = \inf_{\substack{f \in \ell^2(\Tilde{\pi}_t) \\ f \text{ non-constant}}}\left\{  \frac{\mathcal{E}_{K_t^*K_t,\Tilde{\pi}_t}(f)}{\text{Var}_{\Tilde{\pi}_t}(f)} \right\}\,.
\]
Once again, we use:
$$
\forall f \in \ell^2(\Tilde{\pi}_t)\,,\quad \mathcal{E}_{P,u}(f, f)
\leq  2\mathcal{E}_{K_t,\pi_t}(f, f) \leq 4\mathcal{E}_{K_t^*K_t,\pi_t}(f, f)\,.
$$
Moreover, using the minimality of the variance and $\pi_t \leq Mu$ successively, we obtain:
$$ 
\forall f \in \ell^2(\Tilde{\pi}_t)\,,\quad \pi_t(V_N^2) \operatorname{Var}_{\tilde{\pi}_t}(f) \leq  M u(V_N^2) \operatorname{Var}_{\tilde{u}}(f)\,.
$$
We finally find that:
\begin{equation*}
    \forall t \ge 1\,,\quad
\gamma_t \geq  \frac{1}{2M}\gamma_u \,.
\end{equation*}
We recall that, see \cite{Dia}: 
$$\exists \kappa >0\,,\quad \gamma_u\geq \frac{\kappa}{N^2}\,.$$
\textbf{We now give a bound on the merging time .}\\~\\
Let \( T_{\text{mer}} \) be the merging time of the time-inhomogeneous Markov chain driven by the sequence $(K_t)_{t\ge 1}$. We summarize, for all $x\in V_N^2$, the sequence $(\pi_t(x))$ is non-decreasing and:
$$
\forall t\geq 1\,,\,\, \gamma_t \geq \frac{\kappa}{2MN^2}\quad ,\,\forall x \in V_N^2\,,\quad 1 \leq \pi_t(x) \leq M u(x)\,.
$$
Theorem \ref{PoincaréthmMer} yields that there exists a constant $A$ independent of $N,M$ such that:
$$\forall \eta \in (0,1)\,,\quad T_{\text{mer}}(\eta)\le  A M N^2\left(\log(N)+\log(M)+\log(1/\eta)\right) \,.$$
\end{example}
\begin{remark}
    In Example \ref{cyclesexamples}, when two cycles involving the same vertex have different weights at time $t\geq 1$, the measure $\pi_t$ is not reversible with respect to the operator $K_t$, although it remains invariant. Consequently, Example \ref{cyclesexamples} does not constitute an example that involves conductances.
\end{remark}

We present an example with non-monotonic conductances. We study the merging time of a time-inhomogeneous Markov chain which interpolates between two simple random walks on the same set.
\begin{example}[From a simple random walk to another]\label{fromtoanother}
Let $V$ be a finite set. Let $E_1$ and $E_2$ be two distinct sets of edges on $V$. Assume that there exists an edge in \(E_1\) but not in \(E_2\), and vice versa. Define two sets of conductances $c(1)$ and $c(2)$ by respectively:
$$ \forall (x\,,y)\in V \times V\,,\quad c(1)(x,y)=\mathbf{1}_{(x\,,y)\in E_1}\quad \text{and} \quad c(2)(x,y)=\mathbf{1}_{(x\,,y)\in E_2}\,. $$
Let \( P(1) \) and $P(2)$ be the Markov transition operator associated with the set of conductances \( c(1) \) and respectively $c(2)$. Assume that the operators $P(1)$ and $P(2)$ are irreducible on $V$. And, let \( u(1) \) be the reversible measure of $ P(1)$ and \( u(2) \), the one of $P(2)$, see Definition \ref{ElectricNetwork}. We assume that the operators $P(1)$ and $P(2)$ satisfy a Poincar\'e condition:
$$\gamma(P(1)) >0 \quad \text{and}\quad \gamma(P(2)) >0\,.  $$ \begin{figure}[H]
    \centering
    \includegraphics[scale=0.7]{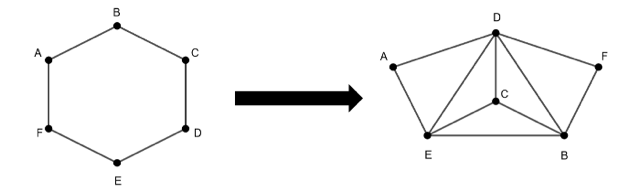}
    \caption{A $6$ vertices example}
    \label{fig: fromto}
\end{figure}
\noindent
Let us consider the sequence of sets of conductances $(c_t)_{t\geq 1}$ defined by:
$$
\forall t\geq 1\,,\quad\forall (x\,,y)\in V \times V\,,\quad c_t(x,y) = \frac{1}{t}c(1)(x,y) + \left(1-\frac{1}{t}\right)c(2)(x,y)\,.
$$
For $t \geq 1$, let $P_t$ be the Markov transition operator on $V$ associated to the set of conductances $c_t$ and let $\pi_t$ be its reversible measure. We study the Markov transition operators $(K_t)_{t\ge 1}$ defined by:
$$\forall t\ge 1\,,\quad K_t =\frac{1}{2}(I+P_t)\,. $$
For $t\geq 1$, the operators $P_t$ and $K_t$ are reversible with respect to $\pi_t$. Let $\gamma_t$ be the Poincar\'e constant $\gamma(K_t^*K_t)$.\\~\\
\textbf{We assume that $P(1)$ and $P(2)$ have constant degree $d_1$ and $d_2$ respectively:} 
$$\forall x\in V\,,\quad \text{card}\left(\{ z  \in V\mid P(1)(x,z)>0\}\right) =d_1 \quad \text{and}\quad \text{card}\left(\{ z  \in V\mid P(2)(x,z)>0\}\right) =d_2\,. $$
Let \(\tilde{u}\) be the uniform probability distribution on \(V\). We get:
$$\tilde{u}(1)=\tilde{u}(2)=\tilde{u} \quad \text{and} \quad \forall t\geq 1\,,\quad \tilde{\pi}_t=\tilde{u}\,. $$
For all $t\geq 1$, the $K_t$'s share the same invariant probability. Therefore, the Markov chain driven by $(K_{t})_{t\geq 1}$ is much simpler to study. More precisely, mixing and merging for this chain coincides and one can apply Theorem \ref{homogeneP}. We estimate the Poincar\'e constant of $K_t^*K_t$. Note that for all $t\geq 1$,
$$ \pi_t =\frac{1}{t}u(1)+\left(1-\frac{1}{t}\right)u(2)\,\, \text{and}\,\, \gamma_t = \inf\left\{  \frac{\mathcal{E}_{K_t^*K_t,\pi_t}(f)}{\text{Var}_{\pi_t}(f)}\mid  f \in \ell^2(\Tilde{\pi}_t) \,,\, f \text{ non-constant}\right\}\,.  $$
For the definition of $\text{Var}_{\pi_t}$, we refer to Equation \ref{defintionvariancenonproba}.\\
A straightforward comparison argument gives:
$$
\forall t\geq 1\,,\quad \gamma_t \geq \gamma_*:= \frac{1}{2}\min\left(\gamma(P(1))\,,\,\gamma(P(2))\right)\,.
$$
For all $x\in V$, let $\mu_t^x$ be the law at time $t$ of the chain started at $x$ and driven by $(K_t)_{t\geq1}$. Theorem \ref{homogeneP} gives:
\begin{equation*}
    \forall x\in V\,,\,\forall t \geq 1\,,\quad d_{TV}( \mu_t^x ,\tilde{u} )\le\frac{1}{\sqrt{\text{card}(V)}}(1-\gamma_*)^{t/2} \,. 
\end{equation*}
\textbf{We do not assume that $P(1)$ and $P(2)$ have constant degree.} We still perform a convex interpolation. We assume that:
$$ \forall x \in V\,,\quad  u(2)(x)\geq u(1)(x)\,.$$
Note that :
$$ \forall t\geq 1\,,\quad \pi_t= u(2) + \frac{1}{t}(u(1)-u(2))\,.$$
Then, for all $x\in V$, the sequence $( \pi_t(x))_{t\geq 1}$ is non-decreasing and once again, a comparison argument gives:
$$
\forall t \ge 1\,,\quad \gamma_t \geq \gamma_*\,.
$$
Besides, the set of conductances $(c_t)_{t\geq 1}$ is not monotone. Indeed, if \((x, y) \in V\) satisfies \(c(1)(x, y) \neq 0\) and \(c(2)(x, y) = 1\), then \((c_t(x, y))_{t \geq 1}\) is non-decreasing. On the other hand, if \((x, y) \in V\) satisfies \(c(2)(x, y) \neq 0\) and \(c(1)(x, y) = 0\), then \((c_t(x, y))_{t \geq 1}\) is non-increasing.\\~\\
Let $A>0$ so that, for all $x \in V$, we have $\tilde{u}(1)(x) \geq  \frac{A}{\text{card}(V)}$. Let \( T_{\text{mer}} \) denote the merging time of the time-inhomogeneous Markov chain driven by the sequence $(K_t)_{t\ge 1}$. An application of Theorem \ref{PoincaréthmMer} yields:\\
$$\forall \eta \in (0,1)\,,\quad T_{\text{mer}}(\eta)\le  \frac{1}{\gamma_*} \left(\log(A)+\log(\text{card}(V))+\log\left(\frac{u(2)(V)}{u(1)(V)}\right)+2\log(1/\eta)\right) \,.$$
\end{example}
We provide an example on an infinite set.
\begin{example}[A birth and death process on $\mathbb{N}$ with extra connections to $0$]
Let $V=\mathbb{N}$. We define the set of edges as $E= \{ (x,y)\mid |x-y| \le 1\} \cup  \cup_{n \geq 1} (0,n) $. Assume that $P$ a positive recurrent, irreducible and aperiodic Markov transition operator on $V$ and $\tilde{u}$ its invariant probability measure are given and satisfy:
$$ \gamma(P)>0\,,\quad\tilde{u}(0)= \sup\left\{\tilde{u}(x) \mid x\in V \right\}\quad \text{and}\quad\forall (x,y)\in V\,,\, |x-y|>1 \implies P(x,y)=0    \,.$$
Let $c_{0}$ be a set of conductance associated to $P$, see Remark \ref{reciproquedefonduct}. In the sequel, we are given:
\begin{itemize}
    \item a sequence of positive integers $(n_t)_{t\ge1}$,
    \item and a sequence of positive reals $(w_t)_{t\ge 1}$ .
\end{itemize} 
Roughly speaking, we want to consider the time-inhomogeneous process defined as follows: at each time $t \ge 1$, we add an undirected edge from $0$ to $n_t$ with a weight proportional to $w_t$.
\begin{figure}[H]
    \centering
    \includegraphics[scale=0.5]{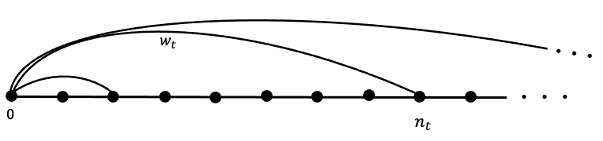}
    \caption{Extra connection with $0$}
    \label{fig: connection}
\end{figure}
\noindent
We build $(c_t)_{t\geq 0}$ by induction a family of non-decreasing sets of conductance where $c_0$ is the one associated to $P$ and for $t\geq 1$, if $c_t$ is given, we define $c_{t+1}$ by:
$$ \forall (x,y)\in V\,,\quad c_{t+1}(x,y)= c_t(x,y) +\Tilde{u}(n_{t+1})w_{t+1}\left(\mathbf{1}_{x=n_{t+1}\,,\,y=0} +\mathbf{1}_{y=n_{t+1}\,,\,x=0} \right)\,. $$
For each $t\geq 1$, let \( P_t \) be the Markov transition operator associated with the set of conductance \( c_t \) and $\pi_t$ be its reversible measure, see Definition \ref{ElectricNetwork}. We study the Markov transition operators $(K_t)_{t\ge 1}$ defined by:
$$\forall t\ge 1\,,\quad K_t =\frac{1}{2}(I+P_t)\,. $$
Let $\gamma_t$ be the Poincar\'e constant $\gamma(K_t^*K_t)$. Note that for each $t\geq 1$, \(\pi_t\) is a reversible measure for \(K_t\) and, for all $x\in V$, the sequence $(\pi_t(x))_{t\geq 1}$ is non-decreasing. Then, the sequence $\{(K_t,\pi_t)\}_{t\ge 1}$ is a finite non-decreasing environment.   \\~\\
We first deal with a \textbf{bounded case}, we assume that: $$M:=\sum_{t\ge 1}w_t   <+\infty \,.$$
It is straight forward that:
$$\forall t\ge 0\,,\forall x\in V\,,\quad \tilde{u}(x) \le\pi_t(x) \le \frac{M + 1}{\tilde{u}(0)}\tilde{u}(x)\,.$$
A comparison argument gives:
$$
\forall t\ge 1\,,\quad\gamma_t \geq  \frac{\tilde{u}(0)}{2(M+1)}\gamma(P)\,.
$$
For any $x\in V$, for each $t\geq 0$, let $\mu_t^x$ be the law of the chain started at $x$ at time $t$ and driven by $(K_{t})_{t\geq1}$. Note that for all $t\geq 1$, we have $\pi_t(V) \le  \frac{M+1}{\tilde{u}(0)}$. An application of Theorem \ref{PoincaréthmMer} yields that for all $x,y$ in $V$,
\begin{equation*}
	\forall t\geq 1\,,\quad d_{TV}( \mu_t^x ,\mu_t^y)\le\frac{1}{2}\sqrt{ \frac{M+1}{\tilde{u}(0)}} \left(\frac{1}{\sqrt{\Tilde{u}(x)}}+\frac{1}{\sqrt{\Tilde{u}(y)}}\right)\left(1-\frac{\tilde{u}(0)}{2(M+1)}\gamma(P)\right)^{t/2}\,.
\end{equation*}
And, we get for all $\eta \in (0,1)$,
$$ T_{\text{mer}}(x,y,\eta) \le \frac{2(M+1)}{\gamma(P)\tilde{u}(0)}\left(\log\left(\frac{M+1}{\tilde{u}(0)}\right)+2\log\left(\frac{1}{\sqrt{\Tilde{u}(x)}}+\frac{1}{\sqrt{\Tilde{u}(y)}}\right)+2\log(1/\eta)\right)\,. $$
We now deal with the \textbf{unbounded case}. We assume that:
\begin{equation}\label{poidsinfini}
    \sum_{t\ge 1}w_t    = +\infty \,.
\end{equation}
For $t\geq 1$, let $M_t$ be $\frac{1+\sum_{r= 1}^tw_r}{\tilde{u}(0)}$. The same calculations as in the \textbf{bounded case} provides:
$$\forall s \in [[1,t]]\,,\quad 
\gamma_s \geq  \frac{1}{2M_t}\gamma(P)\quad \text{and}\quad\pi_s(V) \le M_t\,.
$$
Let $x ,y$ in $V$ and $\eta \in (0,1)$. Note that if $t\geq 1$ is a solution of the following:
\begin{equation}\label{to solve2}
M_t\exp\left(\frac{-t\gamma(P)}{2 M_t}\right) \leq \frac{\eta^2}{\left(\frac{1}{\sqrt{\Tilde{u}(x)}}+\frac{1}{\sqrt{\Tilde{u}(y)}}\right)^2} \,,
\end{equation}
then, Theorem \ref{PoincaréthmMer} gives $T_{\text{mer}}(x,y,\eta) \leq t$. If we have:
$$
\underset{t\rightarrow +\infty}{\liminf} \, \frac{M_t\log(M_t)}{t} = 0 \,,
$$
then, Equation (\ref{to solve2}) has a solution for any choice of $x$, $y$, and $\eta$.\\~\\
We now give \textbf{an always merging condition}. Define $\hat{\gamma}$ the first Dirichlet eigenvalue of $P$ by:
$$
\hat{\gamma}=\inf \left\{\frac{\mathcal{E}_{ P,\tilde{u}}\left(f, f\right)}{\|f\|_{\ell^2(\tilde{u})}^2}\mid f \in  \ell^2(\tilde{u})\,,\, f(0)=0\quad \text{and}\quad f\text { non-constant }  \right\} .
$$
We assume that $\hat{\gamma}>0$, it is sufficient to yield merging in total variation and to find bounds on the merging-time. Below, see Inequation (\ref{donneunnom}). We provide an estimate of $T_{\text{mer}}$ that holds for any choice of $(n_t)_{t\geq 1}$ and $(w_t)_{t \geq 1}$. This gives an estimate of the merging-time even if Condition (\ref{poidsinfini}) is satisfied. This set of examples includes situations when the Markov transition operators $(K_t)_{t\geq 1}$ do not have a limit and merging in total variation still occurs.\\~\\
First, use $\|f - f(0)\|_{\ell^2(\tilde{u})}^2 \geq \text{Var}_{\tilde{u}}(f)$ to find:
$$
 \text{Var}_{\tilde{u}}(f) \leq \|f-f(0)\|_{\ell^2(\tilde{u})}^2
  \leq \frac{1}{\hat{\gamma}}\mathcal{E}_{P,\tilde{u}}(f-f(0),f-f(0)) = \frac{1}{\hat{\gamma}}\mathcal{E}_{P,\tilde{u}}(f,f) \,,
$$
and, we get:
$$
\gamma(P) \geq \hat{\gamma}\,.
$$
For each $t\geq 1$, let $\hat{\gamma}_t$ be the first Dirichlet eigenvalue of $K_t^*K_t =K_t^2$ defined by:
$$
\hat{\gamma}_t=\inf \left\{\frac{\mathcal{E}_{K_t^2,\pi_t}\left(f, f\right)}{\|f\|_{\ell^2(\pi_t)}^2}\mid f \in  \ell^2(\pi_t)\,,\, f(0)=0\quad \text{and} \quad f\text { non-constant } \right\} .
$$
Note that for all $t\geq 1$,
$$
\begin{aligned}
   \|f\|_{\ell^2(\pi_t)}^2&= \|f\|_{\ell^2(\tilde{u})}^2+\sum_{s=1}^t f^2(n_s)w_s(n_s,0)\\
   &\leq  \frac{1}{\hat{\gamma}}\mathcal{E}_{P,\tilde{u}}\left(f, f\right)+ \sum_{s=1}^t \left(f(n_s)-f(0)\right)^2w_s(n_s,0)\\
   &\leq \left(\frac{2}{\hat{\gamma}}+1\right)\mathcal{E}_{K_t^2,\pi_t}(f,f)\,.
\end{aligned}
$$
Let $\gamma:= \frac{1}{(\frac{2}{\hat{\gamma}}+1)}$. We find a uniform lower bound in time on the $\gamma_t$'s:
$$
\forall t\geq 1\,,\quad \gamma_t \geq \gamma\,.
$$
Finally, Theorem \ref{PoincaréthmMer} yields that merging in total variation occurs and for all $x,y$ in $V$,
\begin{equation}\label{donneunnom}
\forall \eta \in (0,1)\,,\quad T_{\text{mer}}(x,y,\eta) \leq \frac{1}{\gamma}\left(\log(1/\eta)+\log\left( \frac{1}{\sqrt{\tilde{u}(x)}}+ \frac{1}{\sqrt{\tilde{u}(y)}}\right)\right)\,.    
\end{equation}
\end{example}
We now illustrate Theorem \ref{MergingNash}.
\begin{example}[Nash Inequalities for a box in $\mathbb{Z}^d$]\label{exampleboite}
Let $N\geq 1$, $d\geq 1$ and $V_N^d$ be $(\mathbb{Z}/N\mathbb{Z})^d$ and let the set of edges $E_N^d$ be $(\mathbb{Z}/N\mathbb{Z})^d\times(\mathbb{Z}/N\mathbb{Z})^d$. Let $c_0$ be the set of conductance defined by:
$$ \forall (x,y)\in (\mathbb{Z}/N\mathbb{Z})^d\,,\quad c_0(x,y) =\mathbf{1}_{|x-y|=1}\,. $$
Let $(c_t)_{t\ge 1}$ be a sequence of set of conductances on $E_N^d$. Assume that:
\begin{itemize} \item for all $ t \geq 1,\, c_t \geq c_0$ and for all $ (x, y) \in E_N^d,\, c_t(x,y) \geq \mathbf{1}_{c_t(x,y) \neq 0}$, \item the sequence $(c_t(x,y))_{t \geq 1}$ is non-decreasing. \end{itemize}
And, we assume that there exists $M>0$ such that:
$$ \forall t\geq 1\,,\, \forall x\in (\mathbb{Z}/N\mathbb{Z})^d\,, \quad  \sum_{y\in V_N^d} c_t(x,y) \le M < +\infty \,.$$
\begin{figure}[H]
    \centering
    \includegraphics[scale=0.5]{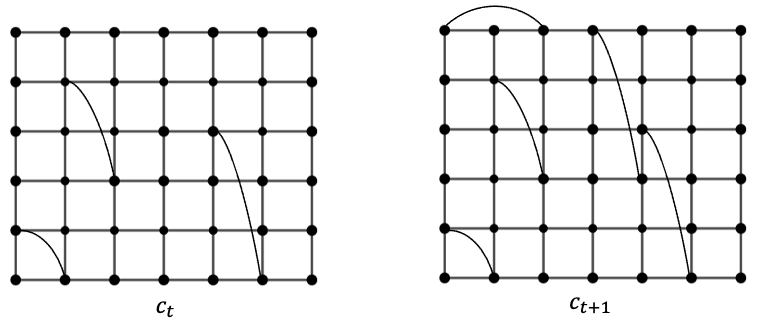}
    \caption{An example in $\mathbb{Z}^2$}
    \label{fig: nash}
\end{figure}
\noindent
For $t \geq 1$, let $P_t$ be the Markov transition operator on $V_N^d$ associated to the set of conductances $c_t$ and let $\pi_t$ be its reversible measure, see Definition \ref{ElectricNetwork}. We study the Markov transition operators $(K_t)_{t\ge 1}$ defined by:
$$\forall t\ge 1\,,\quad K_t =\frac{1}{2}(I+P_t)\,. $$
For each $t\geq 1$, the operators $P_t$ and $K_t$ are reversible with respect to $\pi_t$. Let $Q_t$ be $K^*_tK_t=K_t^2$ and $\gamma_t$ be the Poincar\'e constant $\gamma(Q_t)$.\\~\\
We prove a \textbf{Nash inequality at each time} $t\geq 1$. Let $P$ be the Markov transition operator on $V_N^d$ associated to the set of conductances $c_0$ and let $u$ be its reversible measure. Let $T= dN^2$, it is well known that there exists $\kappa >0$ such that the pair $(P,\Tilde{u})$ satisfies a $\mathcal{N}(\kappa T,d/4,T)$ Nash inequality, see \cite{Dia}.\\~\\
Note that:
$$\forall t \ge 1\,,\quad  2dN^d\le \pi_t(V_N^d) \le N^dM \,. $$
And, we have: 
\begin{equation*}
\forall t \ge 1\,,\quad \frac{2d}{M}\tilde{u} \le \Tilde{\pi}_t \le \frac{M}{2d} \tilde{u} \,.
\end{equation*}
Furthermore,
\[
\begin{aligned}
\forall f : V_N^d \rightarrow \mathbb{R}\,,\, \forall t\geq 1\,,\quad  &\|f\|_{\ell^1(\tilde{u})}^{4/d} \le  \left(\frac{M}{2d}\right)^{4/d} \|f\|_{\ell^1(\Tilde{\pi}_t)}^{4/d}\,, \\
\text{and}\quad &\frac{2d}{M}\|f\|_{\ell^2(\Tilde{\pi}_t)}^2 \le \|f\|_{\ell^2(\tilde{u})}^2 \le \frac{M}{2d} \|f\|_{\ell^2(\Tilde{\pi}_t)}^2\,, \\
\text{and}\quad&\mathcal{E}_{P, u}(f, f) \le \frac{M}{d} \mathcal{E}_{Q_t, \Tilde{\pi}_t}(f, f)\,.
\end{aligned}
\]
Then, using that the pair $(P,\tilde{u})$ satisfies the Nash inequality $\mathcal{N}(\kappa T,d/4,T)$ we find:
$$
\begin{aligned}
    \|f\|_{\ell^2(\Tilde{\pi}_t)}^{2+4/d} &\leq \left(\frac{M}{2d}\right)^{1+2/d} \|f\|_{\ell^2(\tilde{u})}^{2+4/d}\\
    &\leq \kappa T \left(\frac{M}{2d}\right)^{1+2/d}\left( \mathcal{E}_{P, u}(f, f)  +\frac{1}{T} \|f\|_{\ell^2(\tilde{u})}\right) \|f\|_{\ell^1(\tilde{u})}^{4/d}\\
    & \leq \kappa T \left(\frac{M}{2d}\right)^{1+6/d}\left(\frac{M}{d} \mathcal{E}_{Q_t, \Tilde{\pi}_t}(f, f) +\frac{M}{2d}\frac{1}{T}\|f\|_{\ell^2(\Tilde{\pi}_t)}^2  \right)\|f\|_{\ell^1(\Tilde{\pi}_t)}^{4/d}\\
    & \leq 2\kappa T\left(\frac{M}{2d}\right)^{2+6/d}\left( \mathcal{E}_{Q_t, \Tilde{\pi}_t}(f, f) +\frac{1}{T}\|f\|_{\ell^2(\Tilde{\pi}_t)}^2  \right)\|f\|_{\ell^1(\Tilde{\pi}_t)}^{4/d}\,.
\end{aligned}
$$
Finally, let  $\kappa' =2\kappa \left(\frac{M}{2d}\right)^{2+6/d}$ and use a comparison method to prove that the pair $(K_t^2,\Tilde{\pi}_t)$ satisfies the Nash inequality $\mathcal{N}(\kappa' T,d/4,T)$. We get:
\begin{equation*}\label{Nashddim3}
\forall t \ge 1\,,\,\forall f: V_N^d \rightarrow \mathbb{R}\,,\quad \|f\|_{\ell^2(\Tilde{\pi}_t)}^{2+4/d} \leq \kappa' T\left(\mathcal{E}_{K_t^2, \Tilde{\pi}_t}(f, f)+\frac{1}{ T}\|f\|_{\ell^2(\Tilde{\pi}_t)}^2\right)\|f\|_{\ell^1(\Tilde{\pi}_t)}^{4 / d} \,.  
\end{equation*}
We recall that a comparison method gives:
$$\forall t \ge 1\,,\quad \gamma_t \geq  \frac{d}{2M}\gamma(P)\,.$$
It is well known that, see \cite{Dia}: 
$$\exists a >0\,,\quad \gamma(P)=\frac{a}{dN^2}\,.$$
We conclude with:
$$\forall t \ge 1\,,\quad     \gamma_t \geq  \frac{a}{MN^2}\,. $$
We now \textbf{bound the merging time}. For any $t\geq 1$, for any $z\in V_N^d$, denote with $\mu_t^z$ the law of the chain at time $t$ started at $z$ and driven by $(K_{t})_{t\geq1}$. An application of Theorem \ref{MergingNash} yields that there exists a constant $\alpha$ independent of $N$ and $d$ such that for all $\eta \in (0,1)\,,$
$$ \max\left\{ T_{\text{mer}}( x,y,\eta ) \mid (x,y)\in V\right\} \le   \alpha dMN^2\left(\log(1/\eta)+\log(M)\right) \,.$$
And, there exists a constant $\hat{\alpha}$ independent $N$ and $d$ such that:
$$\forall \eta \in (0,1)\,,\quad  T_{mer}^{\infty}( \eta \,, (\pi_t)_{t\geq 1} ) \le   \hat{\alpha}dMN^2\left(\log(1/\eta)+\log(M)\right) \,.$$
\end{example}
the Markov transition operator $K_t$ converges to a Markov transition operator with conductance. We may compare the previous result with the time-homogeneous situation. Let $\hat{K}$ be a Markov transition operator associated with a set of conductance $\hat{c}$ satisfying $1\le \hat{c} \le M/N^d$. Denote by $\hat{\pi}$, the invariant probability of $\hat{K}$, and $\hat{\mu}_t^{x}$, the law of the chain at time $t$ started at $x$ and driven by $\hat{K}^t$. A Nash inequality argument gives the following bound on the mixing time of the chain driven by $\hat{K}$:
$$\exists \tilde{\alpha}> 0\,,\forall \eta \in (0 ,1/2)\,,\quad T_{\text{mix}}^{\infty}(\eta)\le  \tilde{\alpha} M N^2\left(d\log(M)+\log(1/\eta)\right)\,.$$
Then, the estimates on $T_{\text{mix}}^{\infty}$ and on the merging time are similar. The only differences are an additional mass term and in the constants. Besides, the decay factor provided by Nash inequalities cancels out a volume factor in front, thereby providing the correct order of the bound.

\newpage 
We now illustrate Theorem \ref{LogSobolevThm} and  Theorem \ref{LogSobolevThm2}.
\begin{example}[Hypercube with conductance]\label{hypercube}
Let $N\geq 1$ and $V^N$ be $\{0,1\}^{N}$, define $E^N$ the set of edges by $V^N\times V^N$. Let $c_0$ be the set of conductances defined by:
$$ \forall (x,y)\in V^N\,,\quad c_0(x,y) =\mathbf{1}_{|x-y|=1}\,. $$
Let $(c_t)_{t\ge 1}$ be a sequence of set of conductances on $E^N$. Assume that:
\begin{itemize}\item for all $ t \geq 1,\, c_t \geq c_0$ and for all $ (x, y) \in E^N,\, c_t(x,y) \geq \mathbf{1}_{c_t(x,y) \neq 0}$, \item the sequence $(c_t(x,y))_{t \geq 1}$ is non-decreasing. \end{itemize}
And, we assume that there exists $M>0$ such that:
$$ \forall t\geq 1\,,\, \forall x\in V^N\,, \quad  \sum_{y\in V^N} c_t(x,y) \le M < +\infty \,.$$
For $t \geq 1$, let $P_t$ be the Markov transition operator on $V_N^d$ associated to the set of conductances $c_t$ and let $\pi_t$ be its reversible measure, see Definition \ref{ElectricNetwork}. We study the Markov transition operators $(K_t)_{t\ge 1}$ defined by:
$$\forall t\ge 1\,,\quad K_t =\frac{1}{2}(I+P_t)\,. $$
For each $t\geq 1$, the operators $P_t$ and $K_t$ are reversible with respect to $\pi_t$. Let $Q_t$ be $K^*_tK_t=K_t^2$ and $\alpha_t$ be the logarithmic Sobolev constant $\alpha(Q_t)$.
Note that: 
\begin{equation*}
\forall t \ge 1\,,\quad \frac{1}{M}\tilde{u} \le \Tilde{\pi}_t \le M\tilde{u} \,.
\end{equation*}
We \textbf{We give a lower bound on the logarithmic Sobolev constants}. Let $P$ be the Markov transition operator associated to the set of conductances $c_0$. Let $\Tilde{u}$ be its invariant probability, define $K = \frac{1}{2}(I+P)$ and $\alpha =\alpha(K)$, we have, see \cite{SF}:
\begin{equation*}
\exists \kappa >0\,,\quad \alpha = \frac{\kappa}{N}\,.\end{equation*}
For each $t\geq 1$, let $\mathcal{S}_t^N$ the set of functions defined on $V^N$ such that $\mathcal{L}(f\mid \tilde{\pi}_t)$ exists.\\
We prove that:
$$\forall t\geq 1\,,\forall f \in \mathcal{S}_t^N\,,\quad \mathcal{L}(f\mid \tilde{\pi}_t)\leq M \mathcal{L}(f\mid \tilde{u}) \,. $$
Let $t\geq 1$ and $f \in\mathcal{S}_t^N$, we have:
$$
\begin{aligned}
\mathcal{L}(f\mid \Tilde{\pi}_t)&= \inf \left\{ \sum_{x\in V} \tilde{\pi}_t(x) \left( f(x) \log\left(\frac{f(x)}{c}\right)-f(x)+c \right)   \mid c >0 \right\} \quad \text{by Equality \ref{mini2}}\,,\\
&\leq  \sum_{x\in V} \tilde{\pi}_t(x) \left( f(x) \log\left(\frac{f(x)}{\tilde{u}(f) }\right)-f(x)+ \tilde{u}(f)\right)\\
&\leq M   \sum_{x\in V} \tilde{u}(x) \left( f(x) \log\left(\frac{f(x)}{\tilde{u}(f) }\right)-f(x)+ \tilde{u}(f)\right)  \quad \text{using}\quad \Tilde{\pi}_t \le M\tilde{u}\,,\\
&=  M \mathcal{L}(f\mid \tilde{u}) \,.
\end{aligned}
$$
And, note that:
$$ \forall f:V^N \rightarrow \mathbb{R}\,, \quad 2\mathcal{E}_{K_t^2,\Tilde{\pi}_t}(f,f) \geq 2\mathcal{E}_{K_t,\Tilde{\pi}_t}(f,f) \geq \mathcal{E}_{K,\tilde{u}}(f,f)\,. $$
We finally find that the standard logarithmic Sobolev inequality for $K$ gives:
$$\forall t \ge 1\,\quad \forall f\in  \mathcal{S}_t^N\,, \quad  \frac{\mathcal{E}_{K_t^2,\Tilde{\pi}_t}(f,f)}{\mathcal{L}(f^2|\Tilde{\pi}_t)} \ge \frac{1}{2M}\alpha\quad \text{and}\quad   \alpha_t \geq \frac{1}{2M}\alpha \,. $$
\textbf{We bound the merging time}. We are in the conditions to apply Theorems \ref{LogSobolevThm} and \ref{LogSobolevThm2}. Here $K_t=K_t^*$.
Let $s=1+ \left[\frac{MN}{\kappa}\log\left(\log\left(M 2^N\right)\right)\right]$ and $q_{s}=2 (1+\frac{1}{2M}\alpha)^s$. We have:
$$q_{s}\geq 2\prod_{u=1}^s(1+\alpha_{j})\quad \text{and}\quad\log(q_s ) \ge \log(\log(M2^N)) \,.$$
For $t\geq 1$, let $\gamma_t$ be the Poincar\'e constant $\gamma(K_t^*K_t)$ and recall by Lemma \ref{comparaisga}:
$$ \forall t \ge 1\,,\quad \gamma_t\geq 2\alpha_t \quad \text{so}\quad  \gamma_t \geq \frac{\kappa}{MN} \,.$$
For any $t\geq 1$, for any $z\in V^N$, denote with $\mu_t^z$ the law of the chain at time $t$ started at $z$ and driven by $(K_{t})_{t\geq1}$. Theorem \ref{LogSobolevThm} gives that there exists a constant $A$ independent of $N$ such that for all $\eta \in (0,1)\,,$ we have:
$$ \max\left\{ T_{\text{mer}}( x,y,\eta ) \mid (x,y)\in V^N\right\} \le  AMN\left(\log(\log(M)) +\log\left(\log(N)\right) +\log(1/\eta)+\log(M) \right) \,.$$	
And, Theorem \ref{LogSobolevThm2} gives that there exists a constant $B$ independent of $N$ such that:
$$\forall \eta \in (0,1)\,,\quad  T_{\text{mer}}^{\infty}( \eta \,, (\pi_t)_{t\geq 1} ) \le  BMN\left(\log(\log(M)) +\log\left(\log(N)\right) +\log(1/\eta)+\log(M) \right) \,.$$	
\end{example}

In this example, too, the Markov transition operators \( K_t \) have a limit. We find a merging time similar to the mixing time in the time-homogeneous case, up to a factor of \( M \).
\newpage
\section{Proof}\label{proof}\subsection{Preliminaries}  
From now on, we fix $V$ as a countable set. It will be the state space for the Markov chain defined in Section \ref{Introdu}. We will sometimes need to assume that the set $V$ is finite. If this is the case, we will mention it. Let $\mathcal{P}(V)$ be the power set of $V$. Let $\mathcal{M}_{<+\infty}(V)$ be the set of finite measures on $V$ non vanishing and let $\mathcal{M}_{1}(V)$ be the subset of $\mathcal{M}_{<+\infty}(V)$ containing the probability measures on $V$. A function of two variables $K$ is said to be a Markov transition operator on $V$ when $K$ is a map $ V\times V\rightarrow [0\,,1]$ satisfying:
$$\forall x\in V\,,\quad K(x,.):\mathcal{P}(V)\rightarrow[0\,,1]\quad \text{is a probability measure on V .}$$
Let $\mathcal{F}(V,\mathbb{R}):=\left\{f : V \rightarrow \mathbb{R} \right\}$. When $K$ is a Markov transition operator on $V$, the operator $K$ acts as a right linear operator on $\mathcal{F}(V,\mathbb{R})$ by:
\begin{equation}
		\forall f \in\mathcal{F}(V,R)\,,\,\quad\forall x \in V \,,\,(Kf)(x) = \sum_{y \in V} K(x,y)f(y)\,.
\end{equation}
The operator $K$ also acts as an left linear operator on $\mathcal{M}_{<+\infty}(V)$ as follows:
\begin{equation}
\forall \mu \in\mathcal{M}_{<+\infty}(V)\,,\quad \forall y \in V\,,\, 	(\mu K)(y) = \sum_{x \in V} \mu(x)K(x,y) \,.
\end{equation}
Let $\pi \in \mathcal{M}_{<+\infty}(V)$ and $p\in [1,+\infty[$, we define $\ell^p(\pi)\subset\mathcal{F}(V,\mathbb{R})$ by:
$$ \ell^p(\pi) = \left\{ f \in \mathcal{F}(V,\mathbb{R})\mid \sum_{x\in V}\pi(x)|f(x)|^p<+\infty\right\}\,.$$
When $p=1$, we write:
$$ \forall f \in \ell^1(\pi)\,,\quad \pi(f)=\sum_{x\in V}f(x)\pi(x)\quad \text{and }\quad  \Tilde{\pi}(f)=\sum_{x\in V}f(x)\Tilde{\pi}(x)\,.$$ 
We used $\Tilde{\pi}:=\frac{\pi}{\pi(V)}$, it denotes the element of $\mathcal{P}(V)$ given by $\pi$.\\
We consider the following $p$-norm:
$$ \forall f \in \ell^p(\pi)\,,\quad \left\| f \right\| _{\ell^p(\pi)} =\left(\sum_{x\in V}\pi(x)|f(x)|^p\right)^{\frac{1}{p}}\,.$$
In the $\ell^2(\pi)$ case, we write the scalar product by:
\begin{equation*}
\forall (f,g) \in \ell^2(\pi)\,,\quad \left\langle f\mid g\right\rangle_{\pi} = \sum_{x\in V} f(x)g(x)\pi(x)\quad 
    \text{and}\quad \left\langle f\mid g\right\rangle_{\Tilde{\pi}} = \sum_{x\in V} f(x)g(x)\Tilde{\pi}(x) \,.
\end{equation*}
And, the definition of the variance of $f \in \ell^2(\pi)$ linked to $\Tilde{\pi}$ is:
$$
\operatorname{Var}_{\Tilde{\pi}} (f ) = \sum_{x\in V} f^2(x)  \Tilde{\pi}(x) -\left(\sum_{x\in V} f(x)  \Tilde{\pi}(x)\right)^2\,.$$
For \(\pi\), which is not necessarily a probability measure, we use the notation \(\operatorname{Var}_{\pi}\) to denote the following quantity:
$$
\forall f\in \ell^2(\pi)\,,\quad \operatorname{Var}_{\pi} (f )  = \pi(V)\operatorname{Var}_{\Tilde{\pi}} (f ) 
$$
Note that:
$$\begin{aligned}
    \forall f\in \ell^2(\pi)\,,\quad \text{Var}_{\pi}(f)&=\frac{1}{2\pi(V)}\sum_{x\in V}\sum_{y \in V}(f(x)-f(y))^2\pi(x)\pi(y)\\
    &=\sum_{x\in V}f^2(x)\pi(x)-\pi(V)\Tilde{\pi}(f)^2\,.
\end{aligned}$$
Classicaly, we define $\ell^{\infty}(\pi)\subset\mathcal{F}(V,\mathbb{R})$ by the set of bounded functions on $V$ and,
$$ \forall f \in\ell^{\infty}(\pi)\,,\quad  \left\| f\right\| _{\ell^{\infty}(\pi)} = \sup\left\{ \left|f(x)\right| \mid x \in V\right\}\,.$$

One of the characteristics of the time-inhomogeneous context is that the invariant measures of the operators evolve with time. So, we have to consider $K_t$ as an operator between $\ell^p$ spaces with sometimes different measures in the domain and target spaces and other times the same measures. 
\begin{definition}\label{definitionoperator}
Let $t\geq 2$, $\pi_{t-1}$ and $\pi_t$ two positive measures. Let $K_t$ be a Markov transition operator. Assume that $\pi_t$ is an invariant measure of $K_t$. We juggle between:
$$K_t: \ell^p(\pi_t) \rightarrow \ell^q(\pi_t)\quad \text{and}\quad K_t: \ell^p(\pi_t) \rightarrow \ell^q(\pi_{t-1})\,.$$
We denote $K_t^*$ the dual of $K_t:\ell^2(\Tilde{\pi}_t) \rightarrow \ell^2(\Tilde{\pi}_t)$ and  $K_t^{\Rightarrow}=K_{t-1,t}^{\Rightarrow}: \ell^2(\pi_{t-1})\rightarrow\ell^2(\pi_t)$ the dual of $K_t : \ell^2(\pi_{t})\rightarrow\ell^2(\pi_{t-1})$. We use the notation $K_t^{\rightarrow}$ for the dual of the operator $K_t$ from $\ell^2(\Tilde{\pi}_{t})$ to $\ell^2(\Tilde{\pi}_{t-1})$.
\end{definition}
\begin{proposition}
We adopt the notation of Definition \ref{definitionoperator}. We recall that:
$$\forall (x,y)\in V\,,\quad K_t^*(x,y)= \frac{\pi_t(y)}{\pi_t(x)}K_t(y,x) = \frac{\Tilde{\pi}_t(y)}{\Tilde{\pi}_t(x)}K_t(y,x) \,.$$ 
Then, the operator $K_t^*$ is also the dual of $K_t:\ell^2(\Tilde{\pi}_t) \rightarrow \ell^2(\Tilde{\pi}_t)$.
And, we have:
\begin{equation*}
	\forall f \in \ell^2(\pi_t)\,,\,\forall g \in \ell^2(\pi_{t-1})\,,\quad \left\langle K_tf\mid g\right\rangle_{\pi_{t-1}}  = \left\langle f\mid K_t^{\Rightarrow}g\right\rangle_{\pi_{t}} \,.
\end{equation*}
Finally, we get:
$$\forall (x,y)\in V\,,\quad  K_t^{\Rightarrow}(x,y)= \frac{\pi_{t-1}(y)}{\pi_{t}(x)}K_t(y,x) \,.$$
\end{proposition}

\begin{remark}
Note that in general, we can not substitute $\Tilde{\pi}_t$ to $\pi_t$ that is to say if $\pi_t(V)\ne \pi_{t-1}(V)$ then,
$$
\begin{aligned}
\forall (x,y)\in V\,,\quad K_t^{\Rightarrow}(x,y)&\ne \frac{\Tilde{\pi}_{t-1}(y)}{\Tilde{\pi}_{t}(x)}K_t(y,x) \\
\text{but}\quad K_t^{\Rightarrow}(x,y)&= \frac{\pi_{t-1}(V)}{\pi_{t}(V)}\frac{\Tilde{\pi}_{t-1}(y)}{\Tilde{\pi}_{t}(x)}K_t(y,x)\\
\text{and}\quad K_t^{\Rightarrow}(x,y)&= \frac{\pi_{t-1}(V)}{\pi_{t}(V)}K_t^{\rightarrow}(x,y)\,.
\end{aligned}
$$
\end{remark}
\newpage
\subsection{Poincar\'e case}  
Let $\pi \in \mathcal{M}_{<+\infty}(V)$ and $K$ a Markov transition operator on $V$. We assume that $\pi K = \pi$ and let $K^*$ be the dual of $K: \ell^2(\pi)\to\ell^2(\pi)$. We recall that the Dirichlet form subordinated to the measure $\pi$ and the operator $Q:=K^*K$ is defined by:
$$
\begin{aligned}
    \forall (f,g) \in \ell^2(\pi)\,,\quad \hspace{0.1cm}&
    \mathcal{E}_{Q,\pi_t}(f,g) =\left\langle (I-Q)f\mid g\right\rangle_{\pi} \\
    &= \frac{1}{2}\sum_{x \in V}\sum_{y \in V} (f(x)-f(y))(g(x)-g(y)) Q(x,y)\pi(x)\,.
\end{aligned}
$$
\begin{definition}
    In this context, we define $\lambda(Q,\pi)$ by:
\begin{equation}
	1-\lambda(Q,\pi) =\inf\left\{  \frac{\mathcal{E}_{Q,\pi}(f)}{\text{Var}_{\pi}(f)} \mid f \in \ell^2(\Tilde{\pi})\,,\,f \text{ non-constant} \right\}\,.
\end{equation}
We call $1-\lambda(Q,\pi)$ the spectral gap of $Q$ and if $1-\lambda(Q,\pi) >0$, we say that $Q$ satisfies a \textbf{Poincar\'e condition}.
\end{definition}
\begin{remark}
    First, note that \begin{equation}
	1-\lambda(Q,\pi) =1-\lambda(Q,\Tilde{\pi})=\inf\left\{  \frac{\mathcal{E}_{Q,\Tilde{\pi}}(f)}{\text{Var}_{\tilde{\pi}}(f)} \mid f \in \ell^2(\Tilde{\pi})\,,\,f \text{ non-constant} \right\} \,.
\end{equation} 
This quantity is independent of the normalization of $\pi$, so, we forget to specify on it. Furthermore, the constant $\lambda(Q,\pi)$ is the second highest eigenvalue of $Q$.\\
In Sections \ref{Motiv and Back}  and \ref{examples}, we use the notation $\gamma(Q)= 1-\lambda(Q,\pi)$. We have:
\begin{equation}
	\forall f \in \ell^2(\tilde{\pi})\,,\quad (1-\lambda(Q,\pi))\text{Var}_{\pi}(f) \le  \mathcal{E}_{Q,\pi}(f,f)\,.
\end{equation}
\end{remark}

Our main result is the following, it is the key to prove Theorem \ref{PoincaréthmMer} and the proof paves the way to the logarithmic Sobolev case:
\begin{theorem}\label{Poincaréthm}
Let $(K_t)_{t \ge 1}$ be a sequence of Markov transitions operators and $(\pi_t)_{t\ge 1}$ a sequence of $\mathcal{M}_{<+\infty}(V)$. Assume that the sequence $\{(K_t,\pi_t)\}_{t\ge 1}$ is a finite non-decreasing environment. Let $\lambda_t$ be the second highest eigenvalue of $Q_t$.\\
Then,
\begin{equation}\label{Poincaré ineq}
	\forall t \ge 0\,,\,\forall f \in \ell^2(\pi_t)\,,\quad\text{Var}_{\Tilde{\pi}_1}(K_{0,t}f)
	\le \frac{\pi_t(V)}{\pi_1(V)}\text{Var}_{\Tilde{\pi}_t}(f) \prod_{s=1}^t\lambda_s \,. 
\end{equation}
And,
\begin{equation}\label{controlNorm2}
\left\| K_{0,t} -\Tilde{\pi}_1(K_{0,t})\right\| _{\ell^2(\Tilde{\pi}_t)\rightarrow\ell^2(\Tilde{\pi}_1)}
	\le \sqrt{\frac{\pi_t(V)}{\pi_1(V)}\prod_{s=1}^t\lambda_s} \,.
\end{equation}
\end{theorem}

%%%%%%%%%%%%%%%%%%%%%%%%%%%%%%%%%%%POINCARE%%%%%%%%%%%%%%%%%%%%%%%%%%%%%%%%%%%%%%%%%%%%%%%%%%%%%%%%%%%%%%%%%%%%%%%%%%%%%%%%%%%%%%%%%%%%%%%%%%%%%%%%%%%%%%%%%%%%%%%%%%%%%%%%%%%%%%%%%%%%%%%%%%%%%%%%%%%%%%%%%%%%%%%%%%%%%%%%%%%%%%

\begin{remark}\label{formulevariationnelevariance}
The proof of Theorem \ref{Poincaréthm} is not difficult at all. It uses a backward induction and the following:
\begin{equation}\label{mini}
	\forall f \in \ell^2(\pi_t)\,,\quad \text{Var}_{\pi_t}(f)= \inf\left\{ \left\| f-c\right\| ^2_{\ell^2(\pi_t)} \mid c \in \mathbb{R} \right\} \,.
\end{equation}
\end{remark}
%%%%%%%%%%%%%%%%%%%% PREUVE POINCARE%%%%%%%%%%%%%%%%%%%%%%%%%%%%%%%%%
	
\begin{proof}[Proof of Theorem \ref{Poincaréthm}]\label{proof4}
We set $\pi_0$ equal to $\pi_1$. Remark \ref{Inclusion} implies that:
$$\forall t\ge 1\,,\, \forall s \in [[1,t]]\,,\quad \ell^2(\pi_t)\subset \ell^2(\pi_s)\,.$$
\vspace{0,125cm}
Fix $t \ge 1$ and $f\in\ell^2(\pi_t)$. We use the following notation:
$$\forall s \in [[0,t-1]]\,,\quad a_t(s)= \left\| K_{t-s,t}f -\Tilde{\pi}_{t-s}(K_{t-s,t}f)\right\| ^2_{\ell^2(\pi_{t-s})}\,. $$
Given $s\in [[0,t-1]]$, we have:
\begin{equation}\label{aitererr}
    a_t(s+1)\le \lambda_{t-s}a_t(s)\,.
\end{equation}
Indeed, 
\begin{equation*}
\begin{aligned}
	a_t(s+1)&=\left\| K_{t-s}K_{t-s,t}f -\Tilde{\pi}_{t-s-1}(K_{t-s}K_{t-s,t}f)\right\| ^2_{\ell^2(\pi_{t-s-1})}\quad\text{by variance's minimality,}\\
	&\le \left\| K_{t-s}K_{t-s,t}f -\Tilde{\pi}_{t-s}(K_{t-s}K_{t-s,t}f)\right\| ^2_{\ell^2(\pi_{t-s-1})}\\
	&\le\left\| K_{t-s}K_{t-s,t}f -\Tilde{\pi}_{t-s}(K_{t-s}K_{t-s,t}f)\right\| ^2_{\ell^2(\pi_{t-s})} \quad \text{using $\pi_{t-s}\geq\pi_{t-s-1}$,}\\
	&\le \lambda_{t-s}a_t(s) \quad\quad \text{by Poincar\'e Inequality at time }t-s\,.
\end{aligned}
\end{equation*}
By iterating Equation (\ref{aitererr}), a reverse induction argument yields:
\begin{equation}
	a_t(t)=\text{Var}_{\pi_0}(K_{0,t}f) \le  \text{Var}_{\pi_t}(f) \prod_{s=1}^{t}\lambda_{s} =a_t(0) \prod_{s=0}^{t-1}\lambda_{t-s}\,.
\end{equation}
It yields Inequality (\ref{Poincaré ineq}). In order to find Inequality (\ref{controlNorm2}), remark that:
$$\left\|  f -\Tilde{\pi}_t(f)\right\| _{\ell^2(\Tilde{\pi}_t)} \le \left\|  f \right\| _{\ell^2(\Tilde{\pi}_t)}\,.$$
\end{proof}

%%%%%%%%%%%%%%%%%%%%%%%%%%%%%%%%%%%%%%%%%%%%%%%%%%%%%%%%%%%%%%%%%%%%%%%%%%%%%%%%%%%%%%%%%%%%%%%%%%%%%%%%%%%%%%%%%%%%%%%%%%%%%%%%%%%%%%%%%%%%%%%%%%%%%%%%%%%%%%%%%%%%%%%%%%%%%%%%%%%%%%%%%%%%%%%%%%%%%%%%%%%%%%%%%%%%%%%%%%%%%%%
In the study of merging times, a family of operators naturally arises. We need this family to apply Theorem \ref{Poincaréthm} to the study of merging-times. The following result is a Definition-Proposition.
\begin{proposition}\label{Lien}
Let $(K_t)_{t \ge 1}$ be a sequence of irreducible aperiodic Markov transition operators and $(\pi_t)_{t\ge 1}$ a sequence of finite measures. Assume that for $t\geq 1$, $\pi_t$ is an invariant measure of $K_t$. For any $s\geq 0$, $t\geq s$, we define the operator $O_{s,t}$ by:
\begin{equation*}
    \forall f: V\mapsto \mathbb{R}\,,\quad O_{s,t}f:=K_{s,t}f - \Tilde{\pi}_s(K_{s,t}f)\,. 
\end{equation*}
It is an operator from $\ell^{2}(\tilde{\pi}_{t})$ to $\ell^{2}(\tilde{\pi}_{s})$,
i.e.\ $O_{s,t}:\ell^{2}(\tilde{\pi}_{t}) \to \ell^{2}(\tilde{\pi}_{s})$.
\\ 
Let $O_{s,t}^{\rightarrow}$ be the dual of the operator $O_{s,t}$ in the precise spaces: $\ell^2(\Tilde{\pi}_t)\rightarrow\ell^2(\Tilde{\pi}_s)$.\\
Let $\mathbf{1}$ be the function constant equal to 1. For any $z\in V$, for all $t\geq 1$, let $\mu_t^z$ be the law of the chain at time $t$ started at $z$ and driven by $(K_t)_{t\geq1}$. Let $h_t^z =\frac{\mu_t^z}{\Tilde{\pi}_t}$ and $\pi_0=\pi_1$.\\
Then, \begin{enumerate}
    \item For all $t \geq 1$, the family $(O_{s,t})_{s\leq t}$ forms a semigroup.
    \item For any $t \geq s \geq 0$, we have:  
    $$
    \forall f: V\mapsto \mathbb{R}\,,\quad 
    O_{s,t}^{\rightarrow}f = K_{s,t}^{\rightarrow}f - K_{s,t}^{\rightarrow} 1\Tilde{\pi}_s(f)\,.
    $$
    \item Moreover, for all $z \in V$,  
    $$
    O_{0,t}^{\rightarrow}(h_0^z) = h_t^z - K_{0,t}^{\rightarrow} \mathbf{1} \,.
    $$
\end{enumerate}
\end{proposition}

\begin{remark}
Note that for all $t\geq 1$, for all $z\in V$,
\begin{equation*}
h_t^z=K_{0,t}^{\rightarrow} h_0^z\quad\text{and} \quad K_{0,t}^{\rightarrow}\mathbf{1}\Tilde{\pi}_0(h_0^z)= K_{0,t}^{\rightarrow}\mathbf{1}\,.
\end{equation*}
Remark that the quantity $K_{0,t}^{\rightarrow}\mathbf{1}\Tilde{\pi}_0(h_0^z)$ does not depend of $z\in V$, (we recall that $\pi_0=\pi_1$).
\end{remark}
\begin{proof}
We prove \(1.\): the family  $(O_{s,t})_{s\le t}$ is a semi-group.\\
Note that for any $s\le r\le t$, for any $f: V\mapsto \mathbb{R}\,,$
$$ K_{s,r}K_{r,t}f=K_{s,t}f \quad \text{and}\quad \Tilde{\pi}_s(K_{s,r}f)K_{r,t}=\Tilde{\pi}_s(K_{s,t}f)\,, $$
we find:
$$-K_{s,r}\Tilde{\pi}_r(K_{r,t}f)+\Tilde{\pi}_s(K_{s,r}f)\Tilde{\pi}_r(K_{r,t}f)= -\Tilde{\pi}_r(K_{r,t}f)+\Tilde{\pi}_r(K_{r,t}f)=0\,.$$
Then, for any $s\le r\le t$,
$$\begin{aligned}
O_{s,r}O_{r,t}f&=K_{s,r}K_{r,t}f- \Tilde{\pi}_s(K_{s,r}K_{r,t}f) -K_{s,r}\Tilde{\pi}_r(K_{r,t}f)+\Tilde{\pi}_s(K_{s,r}\Tilde{\pi}_r(K_{r,t}f))\\
&= O_{s,t}f\,.
\end{aligned}
$$
We prove \(2.\), for any $t\geq s\geq 0$, we have:
$$O_{s,t}^{\rightarrow}f= K_{s,t}^{\rightarrow}.fK_{s,t}^{\rightarrow}1\Tilde{\pi}_s(f)\,. $$
It is straightforward that:
$$O_{s,t}^{\rightarrow}f=K_{s,t}^{\rightarrow}f - \left(\Tilde{\pi}_s(K_{s,t}.)\right)^{\rightarrow}f\,.$$
We recall that the operators $K_{s,t}^{\rightarrow}$ and $\left(\Tilde{\pi}_s(K_{s,t}.)\right)^{\rightarrow}$ are the dual of the operators $K_{s,t}$ and respectively $\Tilde{\pi}_s(K_{s,t}.)$ from $\ell^2(\Tilde{\pi}_{t})$ to $\ell^2(\Tilde{\pi}_s)$. Note that:
\begin{equation*}
\forall f\in\ell^2(\Tilde{\pi}_t)\,,\,\forall g\in\ell^2(\Tilde{\pi}_s)\,,\quad \left\langle \Tilde{\pi}_s(K_{s,t}f) \mid g \right\rangle_{\Tilde{\pi}_s}= \left\langle f \mid \left(\Tilde{\pi}_s(K_{s,t}.)\right)^{\rightarrow}g  \right\rangle_{\Tilde{\pi}_t} =\Tilde{\pi}_s(K_{s,t}f)\Tilde{\pi}_s(g)\,.
\end{equation*}
And, we get:
\begin{equation*}
\Tilde{\pi}_s(K_{s,t}f)\Tilde{\pi}_s(g)=\left\langle K_{s,t}f \mid \mathbf{1} \right\rangle_{\Tilde{\pi}_s}\Tilde{\pi}_s(g)
=\left\langle f \mid \left(\Tilde{\pi}_s(g)K_{s,t}\right)^{\rightarrow}\mathbf{1} \right\rangle_{\Tilde{\pi}_t}\,.
\end{equation*}
We find that the operator $\left(\Tilde{\pi}_s(K_{s,t}.)\right)^{\rightarrow}:\ell^2(\Tilde{\pi}_s)\rightarrow\ell^2(\Tilde{\pi}_t)$ is equal to the operator $K_{s,t}^{\rightarrow}1\Tilde{\pi}_s(.)$.\\~\\
Finally, we prove \(3.\), for all $z\in V$, for any $t\geq s\geq 0$,
$$O_{0,t}^{\rightarrow}(h_0^z)=h_t^z- K_{0,t}^{\rightarrow}\mathbf{1} \,.$$
We recall that $\pi_0=\pi_1$. We have for $z\in V$, for all $t\geq 1$,
$$\forall y \in V\,,\quad h_t^z(y) = \frac{K_{0,t}(z,y)}{\Tilde{\pi}_t(y)}=   \frac{K_{0,t}^{\rightarrow}(y,z) }{\Tilde{\pi}_0(z)} = K_{0,t}^{\rightarrow} h_0^z(y)\,.$$
Note that:
$$\forall z \in V\,,\quad K_{0,t}^{\rightarrow}\mathbf{1}\Tilde{\pi}_0(h_0^z) = K_{0,t}^{\rightarrow}\mathbf{1}\,.$$ 
Finally, we get:
$$O_{0,t}^{\rightarrow}(h_0^z)= K_{0,t}^{\rightarrow} h_0^z- K_{0,t}^{\rightarrow}\mathbf{1}\,.  $$
\end{proof}

%%%%%%%%%%%%%%%%%%%%%%%%%%%%%%%%%%%%%%%%%%%%%%%%%%%%%%%%%%%%%%%%%%%%%%%%%%%%%%%%%%%%%%%%%%%%%%%%%%%%%%%%%%%%%%%%%%%%%%%%%%%%%%%%%%%%%%BEGIN PREUVE%%%%%%%%%%%%%%%%%%%%%%%%%%%%%%%%%%%%%%%%%%%%%%%%%%%%%%%%%%%%%%%%%%%%%%%%%%%%%%%%%%%%%%%%%%%%%%%%%%%%%%%%%%%%%%%%%%%%%%

In the following, we provide a proof of Theorem \ref{PoincaréthmMer}:
\begin{proof}[Proof of Theorem \ref{PoincaréthmMer}]\label{preuveMergP}
For $z\in V$ and $t\geq 1$, let $\mu_t^z$ the law of the chain started at $z$ and at time $t$ and driven by $(K_{t})_{t\geq1}$. Let $m_t$ be $K_{0,t}^{\rightarrow}\mathbf{1}$ and $\frac{\mu_t^z}{\Tilde{\pi}_t}$ be $h_t^z$.\\
The triangle inequality gives for all $(x,y)\in V$:
$$
d_{TV}(\mu_t^x , \mu_t^y ) \le\frac{1}{2}\left\|  h_t^x -h_t^y\right\| _{\ell^2(\Tilde{\pi}_t)} \le\frac{1}{2}\left\|  h_t^x -m_t\right\| _{\ell^2(\Tilde{\pi}_t)} +\frac{1}{2}\left\|  h_t^y -m_t\right\| _{\ell^2(\Tilde{\pi}_t)} \,.
$$
Then, Proposition \ref{Lien} gives:
$$\forall z \in V\,,\forall t\geq1\,,\quad\left\|  h_t^z-m_t\right\| _{\ell^2(\Tilde{\pi}_t)}^2 \le  \left\| O_{0,t}^{\rightarrow}  \right\| _{\ell^2(\Tilde{\pi}_1)\rightarrow\ell^2(\Tilde{\pi}_t)}^2\left\| h_0^z \right\| _{\ell^2(\Tilde{\pi}_1)}^2 \,.$$
Then, using duality  and Theorem \ref{Poincaréthm}, one finds for all $t\geq 1$,
$$
\begin{aligned}
	\left\| O_{0,t}^{\rightarrow}  \right\| _{\ell^2(\Tilde{\pi}_1)\rightarrow\ell^2(\Tilde{\pi}_t)}^2 &= \left\| K_{0,t}-\Tilde{\pi}_1(K_{0,t})\right\| _{\ell^2(\Tilde{\pi}_t)\rightarrow\ell^2(\Tilde{\pi}_1)}^2 \\ 
	&\le  \frac{\pi_t(V)}{\pi_1(V)}\prod_{s=1}^t \lambda_s\,.
\end{aligned}
$$
Finally, remark that $\left\|  h_0^z\right\| _{\ell^2(\Tilde{\pi}_1)} = \frac{1}{\sqrt{\Tilde{\pi}_1(z)}}$ to conclude that:
\begin{equation*}
	\forall t\geq 1\,,\quad d_{TV}( \mu_t^x ,\mu_t^y ) \le \frac{1}{2}\left(\frac{1}{\sqrt{\Tilde{\pi}_1(x)}}+\frac{1}{\sqrt{\Tilde{\pi}_1(y)}}\right)
	\sqrt{\frac{\pi_t(V)}{\pi_1(V)}\prod_{s=1}^t \lambda_s}\,.
\end{equation*}
\end{proof}

\newpage
\subsection{Nash Inequalities and Poincar\'e} 

Some of the best bounds on $L^2$ mixing times were shown by the use of Nash inequalities. Diaconis and Saloff-Coste introduced them in \cite{Dia} to study mixing. In the time-homogeneous case, a Nash inequality is a tool used to show that when the variance of the density is extremely high, then the walk converges even faster than predicted by the Poincar'e constant. In this part, we show how to use this tool in the time-inhomogeneous case in the non-decreasing environment case.

We aim to prove Theorem \ref{MergingNash}. We first prove Lemma \ref{lemmetempsmelange} and then give a key proposition:\begin{proof}[Proof of Lemma \ref{lemmetempsmelange}]
For $x,y$ in $V$, for $t\geq 1$, we have for any $z\in V$:
\begin{equation*}
\begin{aligned}
      \frac{1}{\pi_{t+1}(V)}\left| \frac{\mu_{t+1}^x(z)-\mu_{t+1}^y(z)}{\Tilde{\pi}_{t+1}(z)}\right|&= \left| \frac{\sum_{u \in V}\left(\mu_{t}^x(u)-\mu_{t}^y(u)\right)K_{t+1}(u,z)}{\pi_{t+1}(z)}\right|\\
      & \leq \max\left\{ \left| \frac{\mu_t^x(v)-\mu_t^y(v)}{\pi_t(v)}\right| \mid v\in V \right\} \sum_{u \in V}\frac{\pi_t(u)K_{t+1}(u,z)}{\pi_{t+1}(z)}\\
      & \leq \max\left\{ \left| \frac{\mu_t^x(v)-\mu_t^y(v)}{\pi_t(v)}\right| \mid v\in V \right\} \sum_{u \in V}\frac{\pi_{t+1}(u)K_{t+1}(u,z)}{\pi_{t+1}(z)} \quad \text{using} \, \pi_t \leq \pi_{t+1}\,,\\
      &= \max\left\{ \left| \frac{\mu_t^x(v)-\mu_t^y(v)}{\pi_t(v)}\right| \mid v\in V \right\} \quad \text{using } \pi_{t+1}K_{t+1} = \pi_{t+1}\,,\\
      &=\frac{1}{\pi_t(V)} s_{\infty}( \mu_t^x \,,\,  \mu_t^y\mid \tilde{\pi}_t)\,.
\end{aligned}
\end{equation*}
By taking the supremum over $z\in V$, we obtain:
$$\forall t\geq 1\,,\quad \frac{1}{\pi_t(V)} s_{\infty}( \mu_t^x \,,\,  \mu_t^y\mid \tilde{\pi}_t) \geq \frac{1}{\pi_{t+1}(V)} s_{\infty}( \mu_{t+1}^x \,,\,  \mu_{t+1}^y\mid \tilde{\pi}_{t+1})\,.$$
\end{proof}
\begin{proposition}\label{controleformdir}
Let $r\geq 1$ and $K_r$ be a Markov transition operator. Let $\pi_{r}$ and $\pi_{r-1}$ two positive finite measures. Assume that $\pi_r$ is an invariant measure for $K_r$. Assume that for all $x$, $\pi_r(x) \geq \pi_{r-1}(x)$. Define $Q_r$ by $K_r^*K_r$.  Then,
$$
\forall g \in \ell^2(\pi_r)\,,\quad \mathcal{E}_{Q_{r},\pi_{r}}(g,g) \le \left\|g - \Tilde{\pi}_{r}(g)\right\|_{\ell^{2}\left(\pi_{r}\right)}^{2} - \left\|K_rg - \Tilde{\pi}_{r-1}(K_rg)\right\|_{\ell^{2}\left(\pi_{r-1}\right)}^{2}\,.
 $$
\end{proposition}
\begin{proof}
Write:
$$ 
\begin{aligned}
\left\|K_rg - \Tilde{\pi}_{r}(K_rg)\right\|_{\ell^{2}\left(\pi_{r}\right)}^{2} &\ge \left\|K_rg - \Tilde{\pi}_{r}(K_rg)\right\|_{\ell^{2}\left(\pi_{r-1}\right)}^{2}\\
&\ge  \left\|K_rg - \Tilde{\pi}_{r-1}(K_rg)\right\|_{\ell^{2}\left(\pi_{r-1}\right)}^{2} \,. 
\end{aligned}
$$
And, 
$$
\begin{aligned}\mathcal{E}_{Q_{r},\pi_{r}}(g,g) &= \left\|g - \Tilde{\pi}_{r}(g)\right\|_{\ell^{2}\left(\pi_{r}\right)}^{2} - \left\|K_rg - \Tilde{\pi}_{r}(K_rg)\right\|_{\ell^{2}\left(\pi_{r}\right)}^{2}\\
&\le \left\|g - \Tilde{\pi}_{r}(g)\right\|_{\ell^{2}\left(\pi_{r}\right)}^{2} - \left\|K_rg - \Tilde{\pi}_{r-1}(K_rg)\right\|_{\ell^{2}\left(\pi_{r-1}\right)}^{2}\,
\end{aligned}
\,. $$
\end{proof}
Nash inequalities allow control of the norms of the operators $(O_{r,t})_{r\leq t}=\left(K_{r,t}-\Tilde{\pi}_r(K_r,t)\right)_{r\leq t}$.
\begin{theorem}\label{ThnNashBorn}
Let $(K_t)_{t \ge 1}$ be a sequence of irreducible aperiodic Markov transition operators on $V$ finite and $(\pi_t)_{t\ge 1}$ a sequence of positive measures. Assume that the sequence $\{(K_t,\pi_t)\}_{t\ge 1}$ is a finite non-decreasing environment. Let $Q_t$ be $K_t^*K_t$. Let $T \geq 1$ and $C, D>0$. Assume that for all $t\geq 1$, the pair $(Q_t,\Tilde{\pi}_t)$  satisfies the $\mathcal{N}(C,D,T)$ Nash inequality. We also assume that $\pi_1(V)\ge 1$. Then, for $0 \leq r \leq t \leq T$,
$$
\left\|K_{r, t}-\Tilde{\pi}_r(K_{r,t}.)\right\|_{\ell^{1}\left(\Tilde{\pi}_{t}\right) \rightarrow \ell^{2}\left(\Tilde{\pi}_{r}\right)}\leq\left(\frac{ C2^{1/D} (1+1 / T)(1+\lceil 4 D\rceil)}{t-r+1}\right)^{D}\frac{\pi_t(V)}{\sqrt{\pi_r(V)}}\,,
$$
and,
$$
\left\|K_{r, t}-\Tilde{\pi}_r(K_{r,t}.)\right\|_{\ell^{2}\left(\Tilde{\pi}_{t}\right) \rightarrow \ell^{\infty}\left(\Tilde{\pi}_{r}\right)}\leq 2\left(\frac{4 C2^{1/D} (1+1 / T)(1+\lceil 4 D\rceil)}{t-r+1}\right)^{D}\sqrt{\pi_t(V)}\,.
$$
\end{theorem}
We state here the Riesz--Thorin interpolation theorem:\begin{theorem}[Riez-Thorin interpolation Theorem]
     Let $1 \leq p_1, q_1 \leq \infty$  and\\  $1 \leq p_1, q_1 \leq \infty$. Assume that $p_1 \leq p_2$ and $ q_1 \leq q_2$.  For any $p \in [p_1 , p_2]$, let $\theta$ and $q \in\left[q_1, q_2\right]$ be:
     $$ 1 / p=\theta / p_1+(1-\theta) / p_2 \quad \text{and}\quad1 / q=\theta / q_1+(1-\theta) / q_2\,. $$ 
    Then, for any operator $K$, we have:
$$
\|K\|_{\ell^p \rightarrow \ell^q} \leq\|K\|_{\ell^{p_1} \rightarrow \ell^{q_1}}^\theta\|K\|_{\ell^{p_2} \rightarrow \ell^{q_2}}^{1-\theta} .
$$
\end{theorem}
\begin{proof}[Proof of Theorem \ref{ThnNashBorn}]
First, note  that for $ s \geq   1$, if the pair $(Q_s,\Tilde{\pi}_s)$  satisfies the $\mathcal{N}(C,D,T)$ Nash inequality and $\pi_s(V) \ge 1$, then, the pair $(Q_s,\pi_s)$  satisfies also  the $\mathcal{N}(C,D,T)$ Nash inequality. Indeed, we have:
$$\forall s \geq1\,,\quad\pi_s(V) \ge 1 \implies\frac{\pi_s(V)^{1+1/2D} }{\pi_s(V)^{1+1/D}} \le 1\,.$$
And, we get:
$$
\forall f: V \rightarrow \mathbb{R}\,, \quad\|f\|_{\ell^2(\pi_s)}^{2+1 / D} \leq C\left(  \mathcal{E}_{Q_s,\pi_s}(f,f) +\frac{1}{T}\|f\|_{\ell^2\left(\pi_s\right)}^2\right)\|f\|_{\ell^1(\pi_s)}^{1 / D} \,. $$
We impose $\pi_0=\pi_1$. Fix $t \in[[1,T]]$ and for $s\in [[0,t]]$, we define:
$$
a_{t}(s)=\left\|K_{t-s, t}f - \Tilde{\pi}_{t-s}(K_{t-s,t} f)\right\|_{\ell^{2}\left(\pi_{t-s}\right)}^{2} .
$$
In the proof of Theorem \ref{Poincaréthm}, it is proven that the sequence $\left(a_{t}(s)\right)_{s\in [[0,t]]}$ is non-increasing.\\~\\
Now, let $f \in \ell^1(\pi_T)$, we apply Nash inequality to the function $K_{t-s,t}f-\Tilde{\pi}_{t-s}(K_{t-s,t}f)$.\\
We get for all $s\in[[0,t-1]] $, 
\begin{equation*}
\begin{aligned}
    &a_{t}(s)^{1+1/(2D)}=\left\|K_{t-s, t}f - \Tilde{\pi}_{t-s}(K_{t-s,t} f)\right\|_{\ell^{2}\left(\pi_{t-s}\right)}^{2(1+1 /(2 D))}\\
	&\le C\left(\mathcal{E}_{Q_{t-s},\pi_{t-s}}(K_{t-s,t}f,K_{t-s,t}f)+a_{t}(s) / T\right)\left\|K_{t-s,t}f-\Tilde{\pi}_{t-s}(K_{t-s,t}f)\right\|_{\ell^1(\pi_{t-s})}^{1/D}\\
	&\le C2^{1/D}\left(\mathcal{E}_{Q_{t-s},\pi_{t-s}}(K_{t-s,t}f,K_{t-s,t}f)+a_{t}(s) / T\right)\left\|K_{t-s,t}f\right\|_{\ell^1(\pi_{t-s})}^{1/D}\\
    &\le C2^{1/D}\left( a_t(s)-a_t(s+1)+a_{t}(s) / T\right)\left\|K_{t-s,t}f\right\|_{\ell^1(\pi_{t-s})}^{1/D}\,.
\end{aligned}
\end{equation*}
To get the last line, we apply Proposition \ref{controleformdir} with $g= K_{t-s,t}f$ and $r=t-s$.\\~\\
So far, we have for all $s\in[[0,t]]$,
\begin{equation}\label{equaNash}
    a_{t}(s)^{1+1/(2D)}
    \le C2^{1/D}\left( a_t(s)-a_t(s+1)+a_{t}(s) / T\right)\left\|K_{t-s,t}f\right\|_{\ell^1(\pi_{t-s})}^{1/D}\,.
\end{equation}
For $s\in[[0,t]]$, let  $n_t(s)=\|K_{t-s,t}f\|_{\ell^1(\pi_{t-s})}$. We prove by induction that this quantities are non-increasing:
\begin{equation*}
\begin{aligned}
\forall s\in [[1,t]]\,,\quad n_t(s)&=\left\|K_{t-s+1}K_{t-s+1,t}f\right\|_{\ell^1(\pi_{t-s})}\\
&\leq \left\|K_{t-s+1}K_{t-s+1,t}f\right\|_{\ell^1(\pi_{t-s+1})} \quad \text{using that} \quad \pi_{t-s}\leq \pi_{t-s+1}\,,\\
&\leq  \left\|K_{t-s+1,t}f\right\|_{\ell^1(\pi_{t-s+1})} \text{using that  } K_{t-s+1}:\ell^1(\pi_{t-s+1}) \rightarrow \ell^1(\pi_{t-s+1}) \text{ is a contraction}\,, \\ 
&=   n_t(s-1)  \,.
\end{aligned}
\end{equation*}
And, we find:
\begin{equation}\label{equa2}
\forall s \in [[0,t]]\,,\quad 	n_t(s) \le n_t(0) = \left\|f\right\|_{{\ell^1(\pi_{t}})}\,.
\end{equation}
 Merge Equations (\ref{equa2}) and (\ref{equaNash}) to finally find:
$$
\forall s \in [[0,t-1]]\,,\quad a_{t}(s)^{1+1 /(2 D)} \leq C 2^{1/D} \left\|f\right\|_{{\ell^1(\pi_{t}})} ^{1/D}\left(a_{t}(s)-a_{t}(s+1)+a_{t}(s) / T\right)\,.
$$
Let $B=B(D, T)= 2^{1/D}(1+1 / T)(1+\lceil 4 D\rceil)$, Lemma 3.1 of \cite{Dia} yields that:
$$
\forall\, 0 \leq r \leq t \leq T\,,\quad a_{t}(t-r) \leq\left(\frac{C   B}{t-r+1}\right)^{2 D}\left\|f\right\|_{{\ell^1(\pi_{t}})} ^{2}\,. 
$$
In particular, we get if $0 \leq r \leq t \leq T$:
$$
\left\|K_{r, t}f-\Tilde{\pi}_r(K_{r,t}f)\right\|_{\ell^{2}\left(\pi_{r}\right)} \leq((C B) /(t-r+1))^{D} \left\|f\right\|_{\ell^{1}\left(\pi_{t}\right)}\,,
$$
and, we find:
$$
\left\|K_{r, t}-\Tilde{\pi}_r(K_{r,t}.)\right\|_{\ell^{1}\left(\Tilde{\pi}_{t}\right)\rightarrow\ell^{2}\left(\Tilde{\pi}_{r}\right)} \leq((C B) /(t-r+1))^{D} \frac{\pi_t(V)}{\sqrt{\pi_r(V)}} .
$$
Let $0 \leq s \leq t \leq T$, we recall that $O_{s,t}=K_{s, t}-\Tilde{\pi}_s(K_{s,t}.)$ and we want to bound the quantity $\left\|O_{s,t}^{\Rightarrow}\right\|_{\ell^{1}\left(\pi_{s}\right) \rightarrow \ell^{\infty}\left(\pi_{t}\right)}$. First, note that duality gives:
$$
\forall\, 0 \leq s \leq t \leq T\,,\quad \left\|O_{s,t}^{\rightarrow}\right\|_{\ell^{2}\left(\Tilde{\pi}_{s}\right) \rightarrow \ell^{\infty}\left(\Tilde{\pi}_{t}\right)} \leq((C B) /(t-s+1))^{D} \frac{\pi_t(V)}{\sqrt{\pi_s(V)}} .
$$
 We define the quantity $M(T)$ by:
$$
M(T)=\max\left\{(t-s+1)^{2 D}\left\|O_{s,t}^{\Rightarrow}\right\|_{\ell^{1}\left(\pi_{s}\right) \rightarrow \ell^{\infty}\left(\pi_{t}\right)}\mid  0 \leq s \leq t \leq T\right\} .
$$
Let $l=\left\lfloor\frac{t-s}{2}\right\rfloor+s$, so that $0 \leq s \leq l \leq t \leq T$. We find:
$$
\begin{aligned}
	\left\|O_{s,t}^{\Rightarrow}\right\|_{\ell^{1}\left(\pi_{s}\right) \rightarrow \ell^{\infty}\left(\pi_t\right)} & \leq\left\|O_{s,l}^{\Rightarrow}\right\|_{\ell^{1}\left(\pi_{s}\right) \rightarrow \ell^{2}\left(\pi_{l}\right)}\left\|O_{l,t}^{\Rightarrow}\right\|_{\ell^{2}\left(\pi_{l}\right) \rightarrow \ell^{\infty}\left(\pi_{t}\right)} \\
	& \leq \left\|O_{s,l}^{\Rightarrow}\right\|_{\ell^{1}\left(\pi_{s}\right) \rightarrow \ell^{2}\left(\pi_{l}\right)}((C B) /(t-l+1))^{D}.
\end{aligned}
$$
Note that for all $0 \leq s \leq l \leq T$:
$$
	\left\|O_{s,l}^{\Rightarrow}\right\|_{\ell^{1}\left(\pi_{s}\right) \rightarrow \ell^{2}\left(\pi_{l}\right)}\leq\left\|O_{s,l}^{\Rightarrow}\right\|_{\ell^{1}\left(\pi_{s}\right) \rightarrow \ell^{\infty}\left(\pi_{l}\right)}^{1 / 2}\left\|O_{s,l}^{\Rightarrow}\right\|_{\ell^{1}\left(\pi_{s}\right) \rightarrow \ell^{1}\left(\pi_{l}\right)}^{1 / 2} .
$$
This follows from the fact that for any function $g$:
$$
	\left\|g\right\|_{\ell^{2}\left(\pi_l\right)} \leq\left\|g\right\|_{\ell^{\infty}\left(\pi_l\right)}^{1 / 2}\left\|g\right\|_{\ell^{1}\left(\pi_l\right)}^{1 / 2} \text {. }
$$
Moreover, note that:
\begin{equation}
    \left\|O_{s,l}^{\Rightarrow}\right\|_{\ell^{1}\left(\pi_{s}\right) \rightarrow \ell^{1}\left(\pi_{l}\right)}= \left\|O_{s,l}\right\|_{\ell^{\infty}\left(\pi_{l}\right) \rightarrow \ell^{\infty}\left(\pi_{s}\right)} \le 2
\end{equation}
So, we get:
$$
\begin{aligned}
	\left\|O_{s,t}^{\Rightarrow}\right\|_{\ell^{1}\left(\pi_{s}\right) \rightarrow \ell^{\infty}\left(\pi_t\right)} & \leq\sqrt{2}\left(\frac{C B}{t-l+1}\right)^{D}\left\|O_{s,l}^{\Rightarrow}\right\|_{\ell^{1}\left(\pi_{s}\right) \rightarrow \ell^{\infty}\left(\pi_l\right)}^{1 / 2} \\
	& \leq\sqrt{2}\left(\frac{C B}{(t-l+1)(l-s+1)}\right)^{D}  M(T)^{1 / 2} \\
	& \leq\sqrt{2}\left(\frac{4 C B}{(t-s+1)^{2}}\right)^{D} M(T)^{1 / 2} .
\end{aligned}
$$
The last inequality follows from the following fact:
$$
t-l+1 \geq \frac{t-s+1}{2} \text { and } l-s+1 \geq \frac{t-s+1}{2} .
$$
Then, we get:
$$M(T) = +\infty\quad \text{or}\quad M(T) \leq 2 (4 C B)^{2 D} \,.$$
However, we assume that $V$ is finite, therefore, $M(T) \leq 2 (4 C B)^{2 D} $ and it follows:
$$
\forall\,0 \leq s \leq t \leq T\,,\quad \left\|O_{s,t}^{\Rightarrow}\right\|_{\ell^{1}\left(\pi_{s}\right) \rightarrow \ell^{\infty}\left(\pi_{t}\right)}\leq 2\left(\frac{4 C B}{t-s+1}\right)^{2 D} .
$$
By duality, we find:
$$
\left\|O_{s,t}\right\|_{\ell^{1}\left(\pi_{t}\right) \rightarrow \ell^{\infty}\left(\pi_{s}\right)}\leq2\left(\frac{4 C B}{t-s+1}\right)^{2 D} .
$$
Let $\theta = 1/2$, $p_1 = 1$, $q_1=q_2=p_2 = \infty$ so that: $$\frac{1}{p_{\theta}}=\frac{\theta}{p_1}+\frac{1-\theta}{p_2}\quad \text{and}\quad\frac{1}{q_{\theta}}=\frac{\theta}{q_1}+\frac{1-\theta}{q_2}\,.$$
Apply the Riesz-Thorin interpolation Theorem to find:
$$
\left\|O_{s,t}\right\|_{\ell^{2}\left(\pi_{t}\right) \rightarrow \ell^{\infty}\left(\pi_{s}\right)}\leq 2\left(\frac{4 C B}{t-s+1}\right)^{ D} .
$$
Finally, we get:
$$
\forall\,0 \leq s \leq t \leq T\,,\quad\left\|O_{s,t}\right\|_{\ell^{2}\left(\Tilde{\pi}_{t}\right) \rightarrow \ell^{\infty}\left(\Tilde{\pi}_{s}\right)}\leq 2\left(\frac{4 C B}{t-s+1}\right)^{ D} \sqrt{\pi_t(V)}.
$$
\end{proof}
We finally provide a proof of Theorem \ref{MergingNash}. This proof follows the same outline as the proof of Theorem \ref{PoincaréthmMer}. We reuse some of its arguments. 
\begin{proof}[Proof of Theorem \ref{MergingNash}]
Let $\pi_0:=\pi_1$ and for all $t\geq 0$, for any $z\in V$, let $\mu_t^z$ be the law at time $t$ of the chain started at $z$ and driven by $(K_{t})_{t\geq1}$.\\
For $t\geq 1$, let $m_t$ be $K_{0,t}^{\rightarrow}1$ and for $z\in V$, and $h_t^z$ be $\frac{\mu_t^z}{\Tilde{\pi}_t}$. The triangular inequality gives:
\begin{equation}
	\forall t\geq 0\,,\,\forall (x,y)\in V\,,\quad d_{TV}(\mu_t^x,\mu_t^y ) \le  \frac{1}{2}\left\| h_t^x -m_t\right\|_{\ell^2(\Tilde{\pi}_t)} +\frac{1}{2}\left\| h_t^y -m_t\right\|_{\ell^2(\Tilde{\pi}_t)}\,. 
\end{equation}
Let $B=B(D, T)=(1+1 / T)(1+\lceil 4 D\rceil)$. By applying Theorem~\ref{ThnNashBorn} to control the $\ell^1 \to \ell^2$ growth and Theorem \ref{Poincaréthm} to bound the spectral contraction, we obtain for all $z \in V$,
$$
\begin{aligned}
\forall s \in [[0,T]]\,,\,\forall t\geq s\,,\quad  
\left\| h_t^z-m_t \right\|_{\ell^2(\Tilde{\pi}_t)} &\leq  \left\|O_{s,t}^{\rightarrow}  \right\|_{\ell^2(\Tilde{\pi}_s)\rightarrow\ell^2(\Tilde{\pi}_t)}\left\|O_{0,s}^{\rightarrow}  \right\|_{\ell^1(\Tilde{\pi}_0)\rightarrow\ell^{2}(\Tilde{\pi}_s)} \left\|h_0^z \right\|_{\ell^{1}(\Tilde{\pi}_0)} \\
&\le 2\sqrt{\pi_t(V)}\left(\frac{4CB}{s+1}\right)^D\prod_{u=s+1}^t \sqrt{\lambda_u} 
\end{aligned} 
$$
Finally, we get:
\begin{equation*}
	\forall s \in [[0,T]]\,,\,\forall t\geq s\,,\quad d_{TV}(\mu_t^x, \mu_t^y) \le2\sqrt{\pi_t(V)}\left(\frac{4CB}{s+1}\right)^D\prod_{u=s+1}^t \sqrt{\lambda_u}\,.
\end{equation*}
In a second step, note that:
$$\forall t\geq 0\,,\quad \max\left\{\left|\frac{\mu_t^x (y) -\tilde{\pi}_1K_{0,t}(y)}{\Tilde{\pi}_t(y)}\right| \mid (x, y) \in V\right\}=\left\|O_{0,t}\right\|_{\ell^1\left(\Tilde{\pi}_{t}\right) \rightarrow \ell^{\infty}\left(\Tilde{\pi}_{0}\right)}\,.
$$
Let $r\leq T$, $u\geq 0$ and $t= 2r+u$, Theorem \ref{Poincaréthm} and Theorem \ref{ThnNashBorn} give:
$$
\begin{aligned}
	\left\|O_{0,t}\right\|_{\ell^1\left(\Tilde{\pi}_{t}\right) \rightarrow \ell^{\infty}\left(\Tilde{\pi}_{0}\right)} 
	\leq &\left\|O_{0,r}^{\rightarrow}\right\|_{\ell^{1}\left(\Tilde{\pi}_{0}\right) \rightarrow \ell^{2}\left(\Tilde{\pi}_{r}\right)} \times\left\|O_{r,r+u}^{\rightarrow}\right\|_{\ell^{2}\left(\Tilde{\pi}_{r}\right) \rightarrow \ell^{2}\left(\Tilde{\pi}_{r+u}\right)} \\
	&\times\left\|O_{r+u,2r+u}^{\rightarrow}\right\|_{\ell^{2}\left(\Tilde{\pi}_{r+u}\right) \rightarrow \ell^{\infty}\left(\Tilde{\pi}_{2r+u}\right)} \\
	&\le 4\pi_t(V)\left(\frac{4 C B}{r+1}\right)^{2D}  \prod_{s=r+1}^{r+u} \sqrt{\lambda_s}\,.
\end{aligned}
$$
Merge the inequalities above to conclude for $r\leq T$, $u\geq 0$ and $t= 2r+u$:
$$
 \max\left\{\left|\frac{\mu_t^x (y) -\tilde{\pi}_1K_{0,t}(y)}{\Tilde{\pi}_t(y)}\right| \mid (x, y) \in V\right\}\le 4\pi_t(V)\left(\frac{4 C B}{r+1}\right)^{2D}  \prod_{s=r+1}^{r+u} \sqrt{\lambda_s}\,.$$
Finally, the triangular inequality gives:
$$ 	\max \left\{ s_{\infty}(\mu_t^x,\mu_t^y \mid \Tilde{\pi}_t) \mid (x,y)\in V\right\} \leq  2 \max\left\{\left|\frac{\mu_t^x (y) -\tilde{\pi}_1K_{0,t}(y)}{\Tilde{\pi}_t(y)}\right| \mid (x, y) \in V\right\} \,.$$
And, for all $\eta \in (0,1)$, we get a bound on $ T_{\text{mer}}^{\infty}( \eta \,, (\pi_t)_{t\geq 1} )$.
\end{proof}
\newpage
\subsection{Logarithmic Sobolev Case}    
In line with the results obtained with functional inequalities for mixing estimates, hypercontractivity is one of the most powerful. These results adapt in a non-decreasing environment to estimate merging for time-inhomogeneous Markov chains.

First, we recall Definitions \ref{definitionentropy} and \ref{definitionconstantelogso}.
For  $\Tilde{\pi}$ be a probability measure on $V$ and  $f$ a function in $\ell^1(\Tilde{\pi})$, we define the entropy of $f$ with respect to $\Tilde{\pi}$ by:
\begin{equation}\label{formuleentropie}
    \mathcal{L}(f\mid \Tilde{\pi}):= \sum_{x\in V} f(x)\log\left(\frac{f(x)}{\Tilde{\pi}(f)}\right)\Tilde{\pi}(x)\,,
\end{equation}
and the standard logarithmic Sobolev constant of $Q$ , $\alpha(Q)$, as the best constant in the inequality: $$\forall f: V \rightarrow \mathbb{R} \,,\quad f- \text{non constant}\,,\quad\frac{\mathcal{E}_{Q,\Tilde{\pi}}(f,f)}{\mathcal{L}(f^2
\mid \Tilde{\pi})} \ge \alpha\, . $$
Note that we have an analogue of Remark \ref{formulevariationnelevariance} and Equation \ref{mini} for the entropy.
\begin{remark}
Fix $\Tilde{\pi}\in \mathcal{M}_{1}(V)$. For any $f$ non-negative, we have:
\begin{equation}\label{mini2}
	  \mathcal{L}(f\mid \Tilde{\pi})= \inf \left\{ \sum_{x\in V} \tilde{\pi}(x) \left( f(x) \log\left(\frac{f(x)}{c}\right)-f(x)+c \right)   \mid c >0 \right\} \,.
\end{equation}
Note that for a non-negative function $f$ and $c> 0$, we have:
$$f(x)\log\left(\frac{f(x)}{c}\right)-f(x)+c \geq  0\,.$$
From (\ref{formuleentropie}), we already know that the infinimum in Equality (\ref{mini2}) is achieved by $c=\Tilde{\pi}(f)$.
\end{remark}
\begin{proposition}\label{propoUtilSobolev}
Let $K$ be an irreducible aperiodic Markov transition operator on $V$ and $\Tilde{\pi}$ its invariant probability. Let $K^*$ be the dual of $K$ in $\ell^2(\Tilde{\pi})$. Set $\alpha:=\alpha(K^*K)$ the logarithmic Sobolev constant of $K^*K$. For all $q\geq 2$, let $q^*= (1+\alpha)q$, then,
\begin{equation}\label{inegalitéhyper}
\forall p\in [1, q^*]\,,\,\forall f \in \ell^q(\Tilde{\pi})\,,\quad\left\|Kf\right\|_{\ell^{p}(\Tilde{\pi})}\le \left\|f\right\|_{\ell^q(\Tilde{\pi})}\,.
\end{equation}
\end{proposition}
Inequality (\ref{inegalitéhyper}) is called a hypercontractivity inequality. For more details, consult \cite{Mi}. Now, we give an analogue of Remark \ref{Inclusion}:
\begin{remark}\label{inclusionentropie}
Let $t\geq 2$, $\pi_{t}$ and $\pi_{t-1}$ two positive finite measures.  Assume that for all $x\in V$, $\pi_t(x) \geq \pi_{t-1}(x)$. Let $\mathcal{S}_t$ be the set of functions such that $\mathcal{L}(f|\Tilde{\pi}_t)$ is well-defined. Then,
\begin{equation}
	\forall t\ge 2\,,\quad \mathcal{S}_{t} \subset \mathcal{S}_{t-1}\,.
\end{equation} 
Indeed, note that for all $ x \in V,$ we have $\Tilde{\pi}_{t-1}(x) \leq \frac{\pi_t(V)}{\pi_{t-1}(V)}   \Tilde{\pi}_{t}(x)$. Then, Equation (\ref{mini2}) gives:
$$
\begin{aligned}
 \mathcal{L}(f\mid \Tilde{\pi}_{t-1}) &\leq  \sum_{x\in V} \Tilde{\pi}_{t-1}(x) \left( f(x) \log\left(\frac{f(x)}{\tilde{\pi}_t(f)}\right)-f(x)+\tilde{\pi}_t(f)\right)  \\ 
 &\leq \frac{\pi_t(V)}{\pi_{t-1}(V)}\sum_{x\in V} \Tilde{\pi}_{t}(x) \left( f(x) \log\left(\frac{f(x)}{\tilde{\pi}_t(f)}\right)-f(x)+\tilde{\pi}_t(f)\right) \\
 &=\frac{\pi_t(V)}{\pi_{t-1}(V)} \mathcal{L}(f\mid \Tilde{\pi}_{t}) 
\end{aligned}
$$

\end{remark}
%%%%%%%%%%%%%%%%%%%%%%%%%%%%%%%%%%%%%HYPER CONTRACTIVITE%%%%%%%%%%%%%%%%%%%%%%%%%%%%%%%%%%%%%%%%%%%%%%%%%%%%%%%%%%%%%%%%%%%%%%%%%%%%%%%%%%%%%%%%%%%%%%%%%%%%%%%%%%%%%%%%%%%%%%%%%%%%%%%%%%%%%%%%%%%%%%%%%%%%%%%%%%%%%%%%%%%%%%%%%%%%%%%%%%%%%%%%%%%%%%%%%%%%%%%%%%%%%%%%%%%%%%%%%%%%%%%%%%%%%%%%%%
\label{Hypercontractity}
To begin with, hypercontractivity adapts to a non-decreasing environment. It takes on a new form. The analogue of Proposition \ref{propoUtilSobolev} is the following:
\begin{theorem}\label{Hypercontractivitythm}
Let $(K_t)_{t \ge 1}$ be a sequence of aperiodic and irreducible Markov transitions operators and $(\pi_t)_{t\ge 1}$ a sequence of positive measures. Assume that the sequence $\{(K_t,\pi_t)\}_{t\ge 1}$ is a finite non-decreasing environment. Let $\alpha_t$ be the logarithmic Sobolev constants $\alpha(K_t^*K_t,\Tilde{\pi}_t)$. For all $q\geq 2$ and $t\geq 1$, define  $q_t$ by:
$$ q_t= q \prod_{s=1}^t(1+\alpha_s) \,.$$\\
Then, for all $q\geq 2$, for all $t\geq 1$,
\begin{equation}\label{HP}
\forall p \in [1,q_t]\,,\,\forall f \in \ell^{q}(\pi_t)\,,\quad \left\|K_{0,t}f\right\|_{\ell^{p}(\Tilde{\pi}_1)}\le \left\|f\right\|_{\ell^{q}(\Tilde{\pi}_t)}\left(\frac{\pi_t(V)}{\pi_1(V)}\right)^{\frac{1}{q}}\,.
\end{equation}
\end{theorem}
\begin{proof}[Proof of Theorem \ref{Hypercontractivitythm}]
We impose$\pi_0:=\pi_1$ and fix $t\ge 1$. Define for $s$ in $[[0,t]]$:
$$q_s =q \prod_{u=t-s+1}^{t}(1+\alpha_u) \quad \text{and}\quad n(s)= \left\|K_{t-s,t}f\right\|_{\ell^{q_s}(\pi_{t-s})}\,.$$ 
Note that for all $s \in [[0,t-1]]$:
$$ \begin{aligned}
\forall g \in \ell^{q_{s+1}}(\pi_{t-s-1})\,,\quad \| K_{t-s} g \|_{\ell^{q_{s+1}}(\pi_{t-s-1})} & \leq   \| K_{t-s} g \|_{\ell^{q_{s+1}}(\pi_{t-s})}\quad  \text{using that } \pi_{t-s-1}\leq \pi_{t-s} \\
&\leq \pi_{t-s}(V)^{\frac{1}{q_{s+1}}-\frac{1}{q_s}} \| g\|_{\ell^{q_{s}}(\pi_{t-s})} \quad \text{by Proposition \ref{propoUtilSobolev}}\,.
\end{aligned} 
$$
Then, we find for all $s \in [[0,t-1]]$:
$$
 n(s+1) \leq \pi_{t-s}(V)^{\frac{1}{q_{s+1}}-\frac{1}{q_s}} n(s)\,.
$$
A straightforward induction gives:
$$
n(t)=\left\|K_{0,t}f\right\|_{\ell^{q_t}(\pi_0)}\le \left\|f\right\|_{\ell^{q}(\pi_t)} \prod_{s=1
}^{t}\pi_{s}(V)^{\frac{1}{q_{t-s+1}}-\frac{1}{q_{t-s}}} =n(0) \prod_{s=0}^{t-1}\pi_{t-s}(V)^{\frac{1}{q_{s+1}}-\frac{1}{q_s}}\,.
$$
Note that:
$$
\begin{aligned}
\prod_{s=1
}^{t}\pi_{s}(V)^{\frac{1}{q_{t-s+1}}-\frac{1}{q_{t-s}}} &= \prod_{s=0
}^{t-1}\pi_{s+1}(V)^{\frac{1}{q_{t-s}}} \prod_{s=1
}^{t}\pi_{s}(V)^{-\frac{1}{q_{t-s}}}\\
&=\pi_{1}(V)^{\frac{1}{q_t}} \pi_{t}(V)^{-\frac{1}{q_0}} \prod_{s=1}^{t-1}\left(\frac{\pi_{s+1}(V)}{\pi_{s}(V)}\right)^{\frac{1}{q_{t-s}}}
\end{aligned}
 $$
Renormalize to find:
\begin{equation*}
\forall t\geq 1\,,\quad \left\|K_{0,t}f\right\|_{\ell^{q_t}(\Tilde{\pi}_0)}\le \left\|f\right\|_{\ell^{q}(\Tilde{\pi}_t)} \prod_{s=1}^{t-1}\left(\frac{\pi_{s+1}(V)}{\pi_{s}(V)}\right)^{\frac{1}{q_{t-s}}}\,.
\end{equation*}
Remark that for all $s\geq 1$, the function $x\to \exp\left( x\ln\left(\frac{\pi_{s+1}(V)}{\pi_{s}(V)} \right)\right)$ is non-decreasing, then,
$$
\prod_{s=1}^{t-1}\left(\frac{\pi_{s+1}(V)}{\pi_{s}(V)}\right)^{\frac{1}{q_{t-s}}} \leq \prod_{s=1}^{t-1}\left(\frac{\pi_{s+1}(V)}{\pi_{s}(V)}\right)^{\frac{1}{q}}= \left(\frac{\pi_t(V)}{\pi_1(V)}\right)^{\frac{1}{q}}\,.
$$
Finally, we get:
\begin{equation*}
\forall t\geq 1\,,\quad \left\|K_{0,t}f\right\|_{\ell^{q_t}(\Tilde{\pi}_0)}\leq\left\|f\right\|_{\ell^{q}(\Tilde{\pi}_t)} \left(\frac{\pi_t(V)}{\pi_1(V)}\right)^{\frac{1}{q}}\,,
\end{equation*}
and the announced result is a consequence of the following:
\begin{equation*}
    \forall p \in [1,q_t]\,,\,\forall h :V \rightarrow \mathbb{R}\,,\quad  \left\|h\right\|_{\ell^{p}(\Tilde{\pi}_0)}\leq \left\|h\right\|_{\ell^{q_t}(\Tilde{\pi}_0)}\,.
\end{equation*}
\end{proof}
In the following, we prove Theorem \ref{LogSobolevThm}  and Theorem \ref{LogSobolevThm2}. The proof follows the same outline as the proof of Theorem \ref{PoincaréthmMer} and we will reuse some of its arguments.
\begin{proof}[Proof of Theorem \ref{LogSobolevThm}  ]\label{Proof2}
Let $\pi_0:=\pi_1$ and for all $t\geq 0$, for any $z\in V$, let $\mu_t^z$ be the law at time $t$ of the chain started at $z$ and driven by $(K_{t})_{t\geq1}$.\\
For $s\geq 1$ and $t\geq s$, let $m_{s,t}$ be $K_{s,t}^{\rightarrow}1$ and for $z\in V$, let $h_t^z$ denote $\frac{\mu_t^z}{\Tilde{\pi}_t}$.\\ 
For $s\geq 1$ and $q_s\ge 2$, we denote $q_s'$ as the Hölder conjugate of $q_s$. For $z\in V$, define $t_z$ by:
\begin{equation}
    s_z=\min \left\{ t \ge 1 \mid  \sum_{u=1}^t \log(1+\alpha_u ) \geq \log\left(\log\left(\Tilde{\pi}_0(z)^{-1}\right)\right)\right\}\,.
\end{equation} 
An application of Proposition \ref{Lien} gives:
$$
\begin{aligned}
	\forall t\ge s\ge 0\,,\,\forall z\in V\,,\quad \left\| h_{t}^z-m_{s,t}\right\|_{\ell^2(\Tilde{\pi}_{t})} &= \left\|O_{s,t}^{\rightarrow}h_s^z\right\|_{\ell^{2}\left(\Tilde{\pi}_{t}\right)}\\
	&\leq\left\|O_{s,t}^{\rightarrow}\right\|_{\ell^{2}\left(\Tilde{\pi}_{s}\right) \rightarrow \ell^{2}\left(\Tilde{\pi}_{t}\right)}\left\|K_{0,s}^{\rightarrow}\right\|_{\ell^{q_s'}\left(\Tilde{\pi}_{0}\right) \rightarrow \ell^{2}\left(\Tilde{\pi}_{s}\right)}\left\|h_0^z\right\|_{\ell^{q_s'}\left(\Tilde{\pi}_{0}\right)}\,.
\end{aligned}
$$
Moreover, for $x,y \in V$, if $s \geq s(x,y)= \max\left\{s_x,s_y\right\}$, then,
$$\left\|h_0^x\right\|_{\ell^{q_s'}\left(\Tilde{\pi}_{0}\right)} = \Tilde{\pi}_0(x)^{-1/q_s} \leq e \quad \text{and}\quad \left\|h_0^y\right\|_{\ell^{q_s'}\left(\Tilde{\pi}_{0}\right)} = \Tilde{\pi}_0(y)^{-1/q_s} \leq e \,. $$
Combine the previous inequalities, Theorem \ref{Poincaréthm} and Theorem \ref{Hypercontractivitythm} to find: 
\begin{equation*}
\begin{aligned}
\forall t \ge s(x,y)\,,\quad d_{TV}(\mu_t^x,\mu_t^y)&\leq \frac{1}{2}\left\| h_{t}^x- m_{s(x,y),t}\right\|_{\ell^2(\Tilde{\pi}_{t})} + \frac{1}{2}\left\| h_{t}^y- m_{s(x,y),t}\right\|_{\ell^2(\Tilde{\pi}_{t})}\\
&\leq e\sqrt{\frac{\pi_t(V)}{\pi_0(V)}}\prod_{u=s(x,y)+1}^t \sqrt{\lambda_u} \,.
\end{aligned}
\end{equation*}
\end{proof}
In the second step, we prove Theorem \ref{LogSobolevThm2} but first, we need the following result:
\begin{lemma}\label{controloperator}
Let $u\geq 2$, $\pi_{u-1}$ and $\pi_u$ two positive measures. Let $K_u$ be a Markov transition operator. Assume that $\pi_u$ is an invariant measure of $K_u$. See the notations in Definition \ref{definitionoperator}. Assume that:
$$ \forall x\in V\,,\quad \pi_u(x)\geq \pi_{u-1}(x)\,.$$
We have:
$$
\forall (a,b) \in [1,\infty]\,,\quad 
\left\|K_{u}^{\Rightarrow}\right\|_{\ell^{a}\left(\pi_{u-1}\right) \rightarrow \ell^{b}\left(\pi_{u}\right)} \leq \left\|K_{u}^*\right\|_{\ell^{a}\left(\pi_{u}\right) \rightarrow \ell^{b}\left(\pi_{u}\right)}\,.
$$
\end{lemma}
\begin{proof}
Denote $a'$ and $b'$ the Hölder conjugate exponents of $a$ and $b$. We get:
$$
\begin{aligned}
\left\|K_{u}^{\Rightarrow}\right\|_{\ell^{a}\left(\pi_{u-1}\right) \rightarrow \ell^{b}\left(\pi_{u}\right)} &= \left\|K_{u}\right\|_{\ell^{b'}\left(\pi_{u}\right) \rightarrow \ell^{a'}\left(\pi_{u-1}\right)}\\
&\leq \left\|K_{u}\right\|_{\ell^{b'}\left(\pi_{u}\right) \rightarrow \ell^{a'}\left(\pi_{u}\right)} \quad \text{using that} \quad \pi_{u-1}\leq \pi_{u}\,,\\
&=\left\|K_{u}^*\right\|_{\ell^{a}\left(\pi_{u}\right) \rightarrow \ell^{b}\left(\pi_{u}\right)}\,.
\end{aligned}
$$    
\end{proof}

\begin{proof}[Proof of Theorem  \ref{LogSobolevThm2}  ]
We recall that $V$ is finite and:
\begin{equation*}
    \forall t \ge 1\,,\quad  \Tilde{\pi}_t^{\sharp} \ge \rho>0 \quad \text{and}\quad \alpha_t \ge \alpha >0  \,.
\end{equation*}
Write \( t \) in the form \( 2r + u \).
Note that:
$$\forall t\geq 0\,,\quad \max\left\{\left|\frac{\mu_t^x (y) -\tilde{\pi}_1K_{0,t}(y)}{\Tilde{\pi}_t(y)}\right| \mid (x, y) \in V\right\}=\left\|K_{0, t}-\Tilde{\pi}_1K_{0,t}\right\|_{\ell^1\left(\Tilde{\pi}_t\right)}=\left\|O_{0,t}\right\|_{\ell^1\left(\Tilde{\pi}_{t}\right) \rightarrow \ell^{\infty}\left(\Tilde{\pi}_{0}\right)}\,.
$$
And, we have with $s=r+u$:
\begin{equation*}
\begin{aligned}
	&\left\|K_{0, t}-\Tilde{\pi}_1K_{0,t}\right\|_{\ell^1\left(\Tilde{\pi}_t\right) \rightarrow \ell^{\infty}\left(\Tilde{\pi}_1\right)}= \left\|O_{0,t}^{\rightarrow}\right\|_{\ell^1\left(\Tilde{\pi}_1\right) \rightarrow \ell^{\infty}\left(\Tilde{\pi}_t\right)}\\
    &\leq\left\|O_{s,t}^{\rightarrow}\right\|_{\ell^{2}\left(\Tilde{\pi}_{s}\right) \rightarrow \ell^{\infty}\left(\Tilde{\pi}_{t}\right)}\times\left\|O_{r,s}^{\rightarrow}\right\|_{\ell^{2}\left(\Tilde{\pi}_{r}\right) \rightarrow \ell^{2}\left(\Tilde{\pi}_{s}\right)}\times\left\|O_{0,r}^{\rightarrow}\right\|_{\ell^{1}\left(\Tilde{\pi}_{1}\right) \rightarrow \ell^{2}\left(\Tilde{\pi}_{r}\right)}\,.
\end{aligned}    
\end{equation*}
We bound these different quantities. First, recall that Theorem \ref{Poincaréthm} states:
$$\left\|O_{r,s}^{\rightarrow}\right\|_{\ell^{2}\left(\Tilde{\pi}_{r}\right) \rightarrow \ell^{2}\left(\Tilde{\pi}_{s}\right)} \le \sqrt{\frac{\pi_s(V)}{\pi_r(V)}\prod_{u=r+1}^s\lambda_u}\,.$$
It remains to bound the two quantities:  $\left\|O_{s,t}^{\rightarrow}\right\|_{\ell^{2}\left(\Tilde{\pi}_{s}\right) \rightarrow \ell^{\infty}\left(\Tilde{\pi}_{t}\right)}$ and $\left\|O_{0,r}^{\rightarrow}\right\|_{\ell^{1}\left(\Tilde{\pi}_{1}\right) \rightarrow \ell^{2}\left(\Tilde{\pi}_{r}\right)}$.\\
The triangle inequality gives:
\begin{align*}
\left\|O_{0,r}^{\rightarrow}\right\|_{\ell^{1}\left(\Tilde{\pi}_{1}\right) \rightarrow \ell^{2}\left(\Tilde{\pi}_{r}\right)} 
&= \left\|K_{0, r}-\Tilde{\pi}_1 K_{0,r}\right\|_{\ell^2\left(\Tilde{\pi}_r\right) \rightarrow \ell^{\infty}\left(\Tilde{\pi}_1\right)} \\
&\leq \left\|K_{0, r}\right\|_{\ell^2\left(\Tilde{\pi}_r\right) \rightarrow \ell^{\infty}\left(\Tilde{\pi}_1\right)} 
   + \left\|\Tilde{\pi}_1 K_{0,r}\right\|_{\ell^2\left(\Tilde{\pi}_r\right) \rightarrow \ell^{\infty}\left(\Tilde{\pi}_1\right)} \\
&\leq 2 \left\|K_{0, r}\right\|_{\ell^2\left(\Tilde{\pi}_r\right) \rightarrow \ell^{\infty}\left(\Tilde{\pi}_1\right)}.
\end{align*}
Then, note that Theorem \ref{Hypercontractivitythm} gives with   $q_r = 2 \prod_{j=1}^r(1+\alpha_j)$:
$$
\begin{aligned}
	 \left\|K_{0,r}^{\rightarrow}\right\|_{\ell^{1}\left(\Tilde{\pi}_{1}\right) \rightarrow \ell^{2}\left(\Tilde{\pi}_{r}\right)}
	\leq& \left\|K_{0,r}\right\|_{\ell^{2}\left(\Tilde{\pi}_{r}\right) \rightarrow \ell^{q_r}\left(\Tilde{\pi}_{1}\right)}
	\left\|I\right\|_{\ell^{q_r}\left(\Tilde{\pi}_{1}\right)\rightarrow \ell^{\infty}\left(\Tilde{\pi}_{1}\right)}\\
	\leq &\sqrt{\frac{\pi_r(V)}{\pi_1(V)}}        \frac{1}{\Tilde{\pi}_1^{\sharp\frac{1}{q_r}}}\,.
\end{aligned}
$$
The last inequality holds by the following:
\begin{equation}\label{derniertrick}
    \forall f \in \ell^{\infty}(\Tilde{\pi}_1)\,,\quad\left\|f\right\|_{\ell^{\infty}(\Tilde{\pi}_1)}\le \frac{1}{\Tilde{\pi}_1^{\sharp\frac{1}{q_r}}} \left\|f\right\|_{\ell^{q_r}(\Tilde{\pi}_1)}\,.
\end{equation}
Choose $r\geq 1$ such that: 
$$2 (1+\alpha)^r \ge \log(1/\rho) \,.  $$
Then, we use the lower bound on the $\Tilde{\pi}_t^{\sharp}$'s and $q_r \geq  2(1+\alpha)^r $ to find:
$$\forall t \ge 1\,,\quad  \frac{1}{\Tilde{\pi}_t^{\sharp\frac{1}{q_{r}}}}\le e \,.  $$
Thus, we get:
$$
 \left\|K_{0,r}^{\rightarrow}\right\|_{\ell^{1}\left(\Tilde{\pi}_{1}\right) \rightarrow \ell^{2}\left(\Tilde{\pi}_{r}\right)}
	\leq e\sqrt{\frac{\pi_r(V)}{\pi_1(V)}}   \,.
$$
We bound $\left\|O_{s,t}^{\rightarrow}\right\|_{\ell^{2}\left(\Tilde{\pi}_{s}\right) \rightarrow \ell^{\infty}\left(\Tilde{\pi}_{t}\right)}$. We have:
$$
\begin{aligned}
	\left\|O_{s,t}^{\rightarrow}\right\|_{\ell^{2}\left(\Tilde{\pi}_{s}\right) \rightarrow \ell^{\infty}\left(\Tilde{\pi}_{t}\right)} &= \left\|O_{s,t}\right\|_{\ell^{1}\left(\Tilde{\pi}_{t}\right) \rightarrow \ell^{2}\left(\Tilde{\pi}_{s}\right)} =  \sup  \left\{ \left\|O_{s,t}f\right\|_{\ell^{2}\left(\Tilde{\pi}_{s}\right)} \mid f \in \ell^{1}\left(\Tilde{\pi}_{t}\right)\,,\, \left\|f\right\|_{\ell^{1}\left(\Tilde{\pi}_{t}\right)}=1  \right\} \\
	&\leq  \sup  \left\{ \left\|K_{s,t}f\right\|_{\ell^{2}\left(\Tilde{\pi}_{s}\right)} \mid  f \in \ell^1(\tilde{\pi}_t)\,,\,\left\|f\right\|_{\ell^{1}\left(\Tilde{\pi}_{t}\right)}=1  \right\} \quad \text{by minimality of variance}\,, \\
	&= \left\|K_{s,t}\right\|_{\ell^{1}\left(\Tilde{\pi}_{t}\right) \rightarrow \ell^{2}\left(\Tilde{\pi}_{s}\right)}\,.
\end{aligned}
$$
Let $\hat{q}(s,t)=2\prod_{u=s+1}^t(1+\hat{\alpha}_u)$. We have by Inequality \ref{derniertrick} with $\Tilde{\pi}_t$ instead of $\Tilde{\pi}_1$:
$$
\begin{aligned}
\left\|K_{s,t}\right\|_{\ell^{1}\left(\Tilde{\pi}_{t}\right) \rightarrow \ell^{2}\left(\Tilde{\pi}_{s}\right)} &\leq \left\|K_{s,t}^{\rightarrow}\right\|_{\ell^{2}\left(\Tilde{\pi}_{s}\right) \rightarrow \ell^{\hat{q}(s,t)}\left(\Tilde{\pi}_{t}\right)} \left\|I\right\|_{\ell^{\hat{q}(s,t)}\left(\Tilde{\pi}_{t}\right) \rightarrow \ell^{\infty}\left(\Tilde{\pi}_{t}\right)}\\
& \le\frac{1}{\Tilde{\pi}_t^{\sharp\frac{1}{\hat{q}(s,t)}}}  \left\|K_{s,t}^{\rightarrow}\right\|_{\ell^{2}\left(\Tilde{\pi}_{s}\right) \rightarrow \ell^{\hat{q}(s,t)}\left(\Tilde{\pi}_{t}\right)}\,.
\end{aligned}
$$
It remains to bound $\left\|K_{s,t}^{\rightarrow}\right\|_{\ell^{2}\left(\Tilde{\pi}_{s}\right) \rightarrow \ell^{\hat{q}(s,t)}\left(\Tilde{\pi}_{t}\right)}$, we define a sequence $(a_k)_{k\in [[s+1,t+1]]}$ by:
$$a_{s+1} =2  \quad \text{and}\quad \forall u \in [[s+2,t]]\,,\quad a_{u+1}=a_u(1+\hat{\alpha}_u) \,.$$
We have:
$$\begin{aligned}
\left\|K_{s,t}^{\Rightarrow}\right\|_{\ell^{2}\left(\pi_{s}\right) \rightarrow \ell^{\hat{q}(s,t)}\left(\pi_{t}\right)} &\leq \prod_{u=s+1}^{t} \left\|K_{u}^{\Rightarrow}\right\|_{\ell^{a_{u}}\left(\pi_{u-1}\right) \rightarrow \ell^{a_{u+1}}\left(\pi_{u}\right)} \\
&\leq\prod_{u=s+1}^{t} \left\|K_{u}^{*}\right\|_{\ell^{a_{u}}\left(\pi_{u}\right) \rightarrow \ell^{a_{u+1}}\left(\pi_{u}\right)} \quad \text{by Lemma \ref{controloperator}}\,,\\
&\leq \prod_{u=s+1}^{t} \pi_{u}(V)^{\frac{1}{a_{u+1}} -\frac{1}{a_{u}}} \quad \text{by Proposition \ref{propoUtilSobolev}}\,.\\
&= \pi_{s+1}(V)^{\frac{1}{a_{s+2}}}\pi_t(V)^{-\frac{1}{a_t}} \prod_{u=s+2}^{t-1} \left(\frac{\pi_{u-1}(V)}{\pi_{u}(V)}\right)^{\frac{1}{a_{u}}}\\
&\leq \pi_{t}(V)^{\frac{1}{a_{t+1}}}\pi_{s+1}(V)^{-\frac{1}{a_{s+2}}} \quad \text{using that } \pi_{u-1}(V)\leq \pi_u(V)\,.
\end{aligned}
$$
Renormalize to find:
\begin{equation*}
    \left\|K_{s,t}^{\rightarrow}\right\|_{\ell^{2}\left(\Tilde{\pi}_{s}\right) \rightarrow \ell^{\hat{q}(s,t)}\left(\Tilde{\pi}_{t}\right)} \leq 1\,,
\end{equation*}
and, using $\hat{q}(s,t) \geq 2(1+\alpha)^r $, we get:
\begin{equation*}
\left\|O_{s,t}^{\rightarrow}\right\|_{\ell^{2}\left(\Tilde{\pi}_{s}\right) \rightarrow \ell^{\infty}\left(\Tilde{\pi}_{t}\right)} \leq e \,.    
\end{equation*}
All the inequalities above merged give:
$$ \max\left\{\left|\frac{\mu_t^x (y) -\tilde{\pi}_1K_{0,t}(y)}{\Tilde{\pi}_t(y)}\right| \mid (x, y) \in V\right\}\le 2e^2\sqrt{\frac{\pi_s(V)}{\pi_1(V)} \prod_{l=r+1}^{r+u} \lambda_{r}}\,. $$
The triangular inequality gives:
$$ 	\max \left\{ s_{\infty}(\mu_t^x,\mu_t^y \mid \Tilde{\pi}_t) \mid (x,y)\in V\right\} \leq  2 \max\left\{\left|\frac{\mu_t^x (y) -\tilde{\pi}_1K_{0,t}(y)}{\Tilde{\pi}_t(y)}\right| \mid (x, y) \in V\right\} \,.$$
And, for all $\eta \in (0,1)$, we get a bound on $ T_{\text{mer}}^{\infty}( \eta \,, (\pi_t)_{t\geq 1} )$.
\end{proof}

\section*{Aknowledgement}
I would first like to thank Pierre Mathieu, my thesis supervisor, without whom this work could not have been initiated or successfully completed, and who also guided me throughout this endeavor. I express gratitude to Laurent Saloff-Coste and Laurent Miclo for their insights on this work. Finally, I extend my thanks to the Alea i2M team members for their support and guidance.

I would also like to thank the referees of the Electronic Journal of Probability, where this article appeared, for their careful reading and helpful comments, which helped improve this paper.

\newpage
\def\refname{REFERENCES}


\begin{thebibliography}{99}
%%%%%%%%%%%%%%%% Martingale %%%%%%%%%%
\bibitem{AM} 
G. Amir, I. Benjamini, O. Gurel-Gurevich and G. Kozma, 
\emph{Random walk in changing environment}, 
arXiv preprint arXiv:1504.04870 (2015).
%%%%%%%%%%%%%%%% Coulhon %%%%%%%%%%%%%%%%%%%%%%%
		
\bibitem{Coulh} T. Coulhon; \emph{Ultracontractivity and Nash type inequalities}. Journal of functional analysis {\bf141} 510--539 (1996).

    
%%%%%%%%%%%%%%%% Dembo %%%%%%%%%%%%%%%%%%%%%%%
		
\bibitem{DAHP}  A. Dembo, R. Huang, B. Morris and Y. Peres;\emph{ Transience in growing subgraphs via evolving sets}. Annales de l'Institut Henri Poincar\'e, Probabilit\'es et Statistiques. Vol. \textbf{53}(3) (2017), 1164--1180. Institut Henri Poincar\'e (2017). 


\bibitem{DAHZ} A. Dembo, R. Huang, and T. Zheng; \emph{Random walks among time increasing conductances: heat kernel estimates}. Probability Theory and Related Fields 175.1  397--445 (2019).
		

%%%%%%%%%%%Diaconis%%%%%%%%%%%%%%%%%%%%%%%%%%%%%%%%%%%%%%
\bibitem{Dia} P. Diaconis and L. Saloff-Coste; \emph{Nash inequalities for finite Markov chains}. Journal of Theoretical Probability vol. 9,  459--510 (1996).
 

%%%%%%%%%%%%%%%%%%%%%%% Huang %%%%%%%%%%%%%%%%%%%%%%%%%%%%%%%%%%%
\bibitem{H} R. Huang; \emph{On random walk on growing graphs}. Annales de l'Institut Henri Poincar\'e, Probabilit\'es et Statistiques. Vol. 55. No. 2. Institut Henri Poincar\'e (2019).


\bibitem{Cycles}S. L.  Kalpazidou; \emph{Cycle representations of Markov processes.} Appl. Math. (N. Y.), 28. Springer-Verlag, New York, (1995).
%%%%%%%%%%%%%%%%%%%%%%%% Livre de Peres %%%%%%%%%%%%%%%%%%
		
\bibitem{Pbook} D. A. Levin and Y. Peres;\emph{ Markov chains and mixing times}. Vol. 107. American Mathematical Soc. (2017).


\bibitem{LPeres} R. Lyons and Y. Peres;\emph{ Probability on trees and networks}. Vol. 42. Cambridge University Press, (2017).
	
		
%%%%%%%%%%%%%%%%%%%% Mathieu %%%%%%%%%%%%%%%%%%%%%
\bibitem{Math} P. Mathieu;\emph{ Carne-Varopoulos bounds for centered random walks.} The Annals of Probability. Vol. 34. nbr. 3. pages. 987--1011 (2006) .



		
%%%%%%%%%%%%%%%%%%%% Miclo %%%%%%%%%%%%%%%%%%%%%%%
\bibitem{Mi} L. Miclo;\emph{ Remarques sur l'hypercontractivit\'e et l'\'evolution de l'entropie pour des cha\^ines de Markov finies}. S\'eminaire de Probabilit\'es, XXXI. Lecture Notes in Math. {\bf1655} 136-167 (1997). 
		
	
%%%%%%%%%%%%%%%%%%%%%%%%%%%%%% MERGING %%%%%%%%%%%%%%%%%%%%%%%%%
\bibitem{SZ1} L. Saloff-Coste and  J. Z{\'u}{\~n}iga;\emph{ Merging for time inhomogeneous finite Markov chains, Part I: Singular values and stability}. Electronic Journal of Probability {\bf14}
1456--1494 (2009).
		
		
\bibitem{SZ2} L. Saloff-Coste and  J. Z{\'u}{\~n}iga;\emph{ Convergence of some time inhomogeneous Markov chains via spectral tecniques}. Stochastic Process. Appl. {\bf117} 961-979 (2007).
		
\bibitem{SZ4} L. Saloff-Coste and J. Z{\'u}{\~n}iga;\emph{ Time inhomogeneous Markov chains with wave like behavior}. Ann. Appl. Probab. {\bf20} 1831-1853 (2010).
		
\bibitem{SZ0} L. Saloff-Coste and J. Z{\'u}{\~n}iga;\emph{ Merging and stability for time inhomogeneous finite Markov Chains}. In Proc. SPA Berlin.  (2009).
		
\bibitem{SZ3} L. Saloff-Coste and J. Z{\'u}{\~n}iga;\emph{ Merging for inhomogeneous finite Markov chains, part II: Nash and log-Sobolev inequalities}. Ann. Probab. {\bf39}
	1161-1203 (2011).
		
		
		
%%%%%%%%%%%%% St Floor %%%%%%%%%%%%%%%%%%%%%%%%%%%
\bibitem{SF} L. Saloff-Coste;\emph{ Lectures on finite Markov chains}. Lectures on Probability Theory and Statistics (Saint-Flour, 1996). Lecture Notes in Math. {\bf1665} 301-413 (1997). 
		

\bibitem{ThomasLuca} S. Thomas and Z. Luca; \emph{Random walks on dynamic graphs: mixing times, hittingtimes, and return probabilities}. arXiv preprint arXiv:1903.01342. (2019).
		
	

\end{thebibliography}
\end{document}